\documentclass[12pt, a4paper]{article}

\usepackage{a4,amsmath,amssymb, amsthm, latexsym, color, graphicx,url, array}
\usepackage{tikz}
\usetikzlibrary{arrows.meta}
\usetikzlibrary{decorations.markings}
\usetikzlibrary{calc}

\usepackage{fullpage}
\usepackage{enumitem}
\usepackage{xcolor}
\usepackage{authblk}
\usepackage{fullpage}
\usepackage{enumitem}
\usepackage{xcolor}
\newtheorem{theorem}{Theorem}[section]
\newtheorem{proposition}[theorem]{Proposition}
\newtheorem{lemma}[theorem]{Lemma}
\newtheorem{corollary}[theorem]{Corollary}

\theoremstyle{definition}
\newtheorem{example}[theorem]{Example}
\newtheorem{definition}[theorem]{Definition}
\newtheorem{remark}[theorem]{Remark}

\title{Cyclotomic skew partial difference sets and partial difference families}
\author[1] {Sophie Huczynska}
\author[2] {Tekg\"{u}l Kalayc\i}
\affil[1]{School of Mathematics and Statistics, University of St Andrews, St Andrews, KY16 9SS, Scotland, UK; email: sh70@st-andrews.ac.uk}
\affil[2]{Institut f\"ur Mathematik, Alpen-Adria-Universit\"at Klagenfurt, Austria; email: tekgulkalayci1@gmail.com}
\date{MSC codes: 05B10, 11T22}
\begin{document}
\maketitle

\begin{abstract}
Skew partial difference sets (skew PDSs) are recently-introduced combinatorial objects closely related to partial difference sets (PDSs). To date, only one construction approach for non-trivial skew PDSs is known, using bent partitions: this produces examples of Latin square type.  In this paper we show that these examples are not an isolated phenomenon; we present new constructions for families of skew PDSs using cyclotomy in finite fields. We provide the first constructions for skew PDSs of Paley type, and new constructions for Latin square type (with different parameters to those from bent partitions).  Moreover, we show how skew PDSs relate to disjoint and external partial difference families (DPDFs/EPDFs), and provide new cyclotomic constructions of both standard and relative DPDFs and EPDFs.
\end{abstract}

\section{Introduction}

Partial difference sets (PDSs) are much-studied combinatorial structures which have close links with strongly regular graphs, two-weight codes and other combinatorial objects \cite{CalKan,Ma84,Ma94,MomWanXia}.  A partial difference set is a subset $A$ of an additively-written group $G$ with the property that the multiset of pairwise differences between distinct elements of the set comprises all non-identity elements of $A$ occurring at one frequency $\lambda$, and all non-identity elements of $G \setminus A$ occurring at another frequency $\mu$.  

The study of partial difference sets forms part of a larger research landscape in which collections of sets are studied in terms of certain regular behaviour of the pairwise differences within the sets (e.g. difference sets \cite{Sto}, almost difference sets \cite{Ara,DinPotWan}, difference families \cite{Bur} and disjoint partial difference families \cite{HucJoh}), or between the sets (e.g. external difference families \cite{ChaDin,PatSti} and external partial difference families \cite{HucJoh}).  Interest in these is motivated both by their inherent combinatorial interest and their wider applications, for example to experiment design, coding theory and cryptography \cite{PatSti}.  Various techniques have been used to construct such structures, including finite geometry, cyclotomy in finite fields, and group theory (among many others).

Recently, in \cite{AnbKalMei}, a new type of difference structure was introduced, called a skew PDS.  This is a set $D \subseteq G$ such that, for some PDS $A \subseteq G$ with $|D|=|A|$, the multiset of all pairwise differences between distinct elements of $D$ comprises $\lambda$ occurrences of non-identity elements of $A$ and $\mu$ copies of the non-identity elements of $G \setminus A$.  In other words, $D$ has the internal difference profile of $A$, but need not be $A$ itself (indeed the case where $D$ is $A$ or an additive translate of $A$ is viewed as a trivial case).  Such a $D$ is said to be a skew PDS corresponding to the PDS $A$.

In \cite{AnbKalMei}, the first construction method for non-trivial skew PDSs was presented.  The approach uses bent partitions of elementary abelian groups.  Bent partitions were introduced recently in \cite{AnbMei}, and have previously been used to obtain partial difference sets and LP-packings.  The skew PDSs obtained correspond to PDSs of Latin square type.

In this paper, we demonstrate that skew PDSs can be constructed in other ways, and that there exist families of skew PDSs corresponding to types other than Latin square type.  We present the first cyclotomic constructions for skew PDSs: this includes the first constructions for skew PDSs of Paley type, as well as constructions for skew PDSs of Latin square type with parameters different to those of \cite{AnbKalMei}.  We note the connection between skew PDSs and disjoint and external partial difference families relative to PDSs; arising from this link, we exhibit constructions for relative DPDFs and EPDFs, answering an Open Question of \cite{HucJoh}.  These are relative to Paley PDSs. Moreover, we show that some of the approaches used here to obtain relative DPDFs and EPDFs can be applied to obtain new constructions for standard DPDFs and EPDFs which partition Paley PDSs. The work was underpinned by computational search for examples using GAP \cite{GAP4}.

\section{Background}

Throughout, let $G$ be a group, written additively.  For any $A \subseteq G$, denote the set of its non-identity elements by $A^*=A \setminus \{0\}$.

\subsection{Partial difference sets and skew partial difference sets}

For a subset $D$ of $G$, define the multiset 
$$\Delta(D)=\{ x-y: x \neq y \in D \}$$
and for sets $D_1, D_2 \subseteq G$, define the multiset 
$$\Delta(D_1, D_2)=\{ x-y: x \in D_1, y \in D_2 \}.$$  

For two multisets $M$ and $N$, denote by $M+N$ the multiset in which the elements of $M$ and $N$ appear with the sum of their multiplicities in $M$ and $N$.  For $\lambda \in \mathbb{N}$, denote by $\lambda M$ the multiset comprising $\lambda$ copies of $M$.

\begin{definition}
Let $G$ be a group of order $v$, and let $A$ be a subset of $G$ of size $k$. Then $A$ is a $(v,k,\lambda,\mu)$ \emph{partial difference set} in $G$ if the following multiset equation holds:
$$\Delta(A)=\lambda A^* + \mu (G^*\setminus A).$$
\end{definition}

If a PDS $A$ satisfies $0 \neq A$ and $A=\{-a: a \in A \}=-A$, it is called \emph{regular}.  We have the following results for PDSs \cite{HucJoh, Ma94}:
\begin{proposition}
Let $A$ be a $(v,k,\lambda,\mu)$-PDS.
\begin{itemize}
\item[(i)] If $\lambda \neq \mu$ (i.e., $A$ is not a DS) then $A=-A$.
\item[(ii)] If $A=-A$ then all of the following are PDSs: $A \setminus \{0\}$, $A \cup \{0\}$, $G \setminus A$, $(G \setminus A) \setminus \{0\}$ and $(G \setminus A) \cup \{0\}$.
\end{itemize}
\end{proposition}

The following definition is introduced in \cite{AnbKalMei}:
\begin{definition}
Let $G$ be an abelian group of order $v$, and let $D$ be a subset of $G$ of size $k$. Then $D$ is a $(v,k,\lambda,\mu)$ \emph{skew partial difference set} in $G$ if there exists a $(v,k,\lambda,\mu)$ partial difference set $A \subseteq G$ such that the following multiset equation holds:
$$\Delta(D)=\lambda A^* + \mu (G^*\setminus A).$$
\end{definition}
We will say that $D$ is a skew PDS \emph{corresponding to} $A$, and if $A$ is a PDS of a particular type, we will say that $D$ is a skew PDS of this type.

We have the following trivial examples, noted in \cite{AnbKalMei}:
\begin{example}
\begin{itemize}
\item[(i)] Any $(v,k,\lambda,\mu)$ partial difference set $A \subseteq G$ is a skew PDS corresponding to $A$.
\item[(ii)] Any additive translate $a+A$ of $(v,k,\lambda,\mu)$ partial difference set $A \subseteq G$ (where $a \in G$) is a skew PDS corresponding to $A$.
\end{itemize}
\end{example}

The following results are established in \cite{AnbKalMei}: 
\begin{theorem}\label{thm:skewPDScomp}
Let $D$ be a skew PDS in an abelian group $G$.
\begin{itemize}
\item[(i)] If $\Delta(D)=\lambda A+ \mu(G^* \setminus A)$ then $A$ and $A \cup \{0\}$ are PDSs.  Exactly one of them has the same parameters as $D$ and is the PDS corresponding to $D$. 
\item[(ii)] If $D$ is a skew PDS corresponding to the $(v, k,  \lambda, \mu)$-PDS $A$, then $G \setminus D$ is also a skew PDS corresponding to  the $(v, v- k, v -2k  + \mu, v -2k + \lambda)$-PDS $G \setminus A$. 
\end{itemize}
\end{theorem}

\begin{remark}
Let $D$ be a skew PDS corresponding to the PDS $A$. We have the following ``negative" results.   The property $D=-D$ does not necessarily hold. $G^* \setminus D$ is not a skew PDS in general.  If $0 \notin D$, then $D \cup \{0\}$ is not necessarily a skew PDS; moreover if $0 \in D$, then $D \setminus \{0\}$ is not necessarily a skew PDS. 
\end{remark}

\subsection{PDSs from cyclotomic classes}

Let $\mathbb{F}_q$ be the finite field of $q$ elements where $q$ is a prime power, and let $\alpha$ be a primitive element of $\mathbb{F}_q$.  Let $q=ef+1$.  Then the \emph{cyclotomic classes of order e} in $\mathbb{F}_q$ are the multiplicative subgroups
$$ C_i^e= \alpha^i \langle \alpha^e \rangle$$
where $0 \leq i \leq e-1$. Each class has cardinality $f$.

It is well-known that difference sets and partial difference sets can be obtained by taking single cyclotomic classes, or unions of such classes, under certain conditions.  The first such result is due to Paley.

\begin{proposition}
Let $q$ be an odd prime power.  In $\mathbb{F}_q$,    the set of non-zero squares $C_0^2$ forms a:
\begin{itemize}
\item[(i)] $(q,(q-1)/2,(q-5)/4,(q-1)/4)$-PDS when $q \equiv 1 \mod 4$;
\item[(ii)] $(q,(q-1)/2,(q-3)/2)$-DS when $q \equiv 3 \mod 4$.
\end{itemize}
This is also true for the set of non-zero nonsquares $C_1^2$.
\end{proposition}

If $q \equiv 1 \mod 4$, then $-1$ is a square in $\mathbb{F}_q$, so the PDSs $C_0^2$ and $C_1^2$ are both regular.  

\begin{definition}
A regular $(v,(v-1)/2,(v-5)/4,(v-1)/4)$-PDS with $v \equiv 1 \mod 4$ is called a \emph{Paley} type PDS.
\end{definition}

Another type of PDS comprising a single cyclotomic class of a finite field is the following (see Theorem 4.3 of \cite{HucJoh}).

\begin{theorem}\label{thm:C_0^4PDS}
Let $q=4f+1$ be a prime power, such that $q=x^2+4y^2$ is the proper representation of $q$ with $x \equiv 1 \mod 4$.  Then $C_0^4$ is a PDS if and only if $y=0$.
The parameters are $(q,(q-1)/4,(q-11-6x)/16,(q-3+2x)/16)$.
This holds precisely when $q=p^m$, $p \equiv 3 \mod 4$ and $m$ is even.
\end{theorem}

\begin{definition}
\begin{itemize}
\item[(i)] A PDS with parameters $(n^2,r(n-1),n+r^2-3r,r^2-r)$ is said to be of \emph{Latin square} type.
\item[(ii)] A PDS with parameters $(n^2,r(n+1),-n+r^2+3r, r^2+r)$ is said to be of \emph{negative Latin square} type. 
\end{itemize}
\end{definition}

When $q=p^m$ in Theorem \ref{thm:C_0^4PDS} with $p \equiv 3 \mod 4$ and $m/2$ an odd integer, then $C_0^4$ is a Latin square type PDS with parameters $(n^2,r(n-1),n+r^2-3r,r^2-r)$ where $n=p^{m/2}$ and  $r=(p^{m/2}+1)/4$.  Here $x=-p^{m/2}$.
If $m/2$ is an even integer, then $C_0^4$ is of negative Latin square type PDS with parameters $(n^2,r(n+1),-n+r^2+3r, r^2+r)$ where $n=p^{m/2}$ and  $r=(p^{m/2}-1)/4$.  Here $x=p^{m/2}$.  In \cite{AnbKalMei}, Latin square type skew PDSs in $\mathbb{F}_{p^m} \times \mathbb{F}_{p^m}$ are presented, where $n=p^m$, $r=p^{m-k}$ and $k$ is a divisor of  $m$.

\begin{example}
\begin{itemize}
\item[(i)] For $q=9=3^2$, $C_0^4$ is a $(9,2,1,0)$-PDS of Latin square type.
\item[(ii)] For $q=81=3^4$, $C_0^4$ is an $(81,20,1,6)$-PDS of negative Latin square type.
\item[(iii)] For $q=361=19^2$, $C_0^4$ is a $(361,90,29,20)$-PDS of Latin square type.
\end{itemize}
\end{example}

There are results in the literature stating conditions under which various $C_0^e$ ($e \leq 8$) form DSs or PDSs \cite{HucJoh, Ma84, Ma94, MomWanXia, Sto}.  It is conjectured that all PDSs formed of a single cyclotomic class take one of three forms: subfield, semi-primitive case or sporadic \cite{MomWanXia}.

It is straightforward to show that no single cyclotomic class of given order can be a skew PDS corresponding to a PDS that is a cyclotomic class of the same order. 

\begin{proposition}\label{prop:noC_iskewPDS}
For any proper partial difference set $A \subseteq \mathbb{F}_q$ of the form $C_0^e$ for some $e|q-1$, there is no $C_i^e$ ($0<i<e$) which is a skew PDS corresponding to $A$.
\end{proposition}
\begin{proof}
Since it is a proper PDS,
$$ \Delta(C_0^e)= \lambda C_0^e + \mu (\cup_{j \neq 0}C_j)$$
where $\lambda \neq \mu$.  Since
$$ \Delta(C_i^e)= \lambda C_i^e + \mu (\cup_{j \neq 0} C_{i+j})$$
it is not possible for the multiset $\Delta(C_i^e)$ to equal $\Delta(C_0^e)$ when $0<i<e$.
\end{proof}

One way to view the above result is that, if $C_0^e$ is a PDS, then $C_i^e$ is a PDS for all $0<i<e$.  Results in the literature which guarantee that cyclotomic classes form PDSs are in some sense ``in tension" with attempts to construct skew PDSs.  In a similar vein,  the following useful result \cite{CalKan, HucJoh, Ma94}, on constructing PDSs from unions of cyclotomic classes,  again rules-out the possibility that such unions could form skew PDSs.

\begin{theorem}
Let ${\rm GF}(q')$ be a finite field of order $q^{\prime} = q^{2\beta}$, where $\beta \in \mathbb{N}$ and $q$ is a power of a prime p. Let $e\mid{q+1}$ and set $\eta = \left(\frac{(-q)^{\beta}-1}{e}\right)$. For any $I \subset \{0,1,\ldots,e-1\}$, ($|I|=u, 2 \leq u \leq e-1$) where $\mathcal{D}^{\prime} = \{C_i^e\}_{i \in I}$, 
\begin{itemize}
\item[(i)] each $C_i^{e}$ is a (regular) $(q^{\prime},\frac{q^{\prime}-1}{e},\eta^2-(e-3)\eta-1,\eta^2+\eta)$-PDS;
\item[(ii)]  $D=\cup_{i \in I} C_i^e$ is a (regular) $(q^{\prime}, \frac{u(q^{\prime}-1)}{e}, u^2 \eta^2+(3u-e)\eta -1, u^2 \eta^2+u \eta)$-PDS, which is proper except when $\eta=1=2u-e$ or $\eta=-1=2u-e$.
\end{itemize}
\end{theorem} 

As before, let $q=ef+1$ where $q$ is a prime power.  For each pair of cyclotomic classes $C_i^e$ and $C_j^e$, the cyclotomic number $(i,j)_e$ is defined to be the number of solutions to 
$$ z_i+1=z_j$$
where $z_i \in C_i^e$ and $z_j \in C_j^e$.

Cyclotomic numbers have been an object of study since Gauss \cite{Gau}, and can be useful in establishing properties of difference structures formed from cyclotomic classes. However, there are various limitations to their use, including the sign ambiguity associated with them, and the fact that they are known only for certain values of $e \leq 22$. In situations when cyclotomic number results are available, the following well-known result can be extremely useful (for a proof, see \cite{HucJoh}). 

\begin{lemma}\label{lemma:cyc}
Let $q=ef+1$ be a prime power.  For $0 \leq j,l \leq e-1$:
\begin{itemize}
\item[(i)] $\Delta(C_j^e)=\sum_{i=0}^{e-1} (i,0)_e C_{i+j}^e$
\item[(ii)] $\Delta(C_{j+l}^e,C_l^e)=\sum_{i=0}^{e-1} (i,j)_e C_{i+l}^e$.
\end{itemize}
\end{lemma}

\section{Cyclotomic skew PDSs using classes of order $4$}

In this section, we present the first construction for skew Paley PDSs, using cyclotomic classes of order $4$.  The necessary information regarding cyclotomic numbers of order $4$ may be found in the Appendix.  

\begin{theorem}\label{thm:C0C3}
Let $q$ be a prime power congruent to $5$ modulo $8$. Let $\alpha$ be a primitive element of $\mathbb{F}_q$. Let $q=s^2+t^2$ be the unique proper representation (with sign ambiguity resolved) as in Theorem \ref{thm:KaRaThm}. 

Then $D=C_0^4 \cup C_3^4$ is a Paley skew PDS corresponding to 
\begin{itemize}
\item[(i)] the Paley $(q,\frac{q-1}{2}, \frac{q-5}{4}, \frac{q-1}{4})$-PDS $C_0^2$, if $t=-2$;
\item[(ii)] the Paley $(q,\frac{q-1}{2}, \frac{q-5}{4}, \frac{q-1}{4})$-PDS $C_1^2$, if $t=2$.
\end{itemize}
\end{theorem}
\begin{proof}
Since $D=C_0^4 \cup C_3^4$, we have $\Delta(D)=\Delta(C_0^4)+\Delta(C_3^4)+\Delta(C_0^4,C_3^4)+\Delta(C_3^4,C_0^4)$.
By Lemma \ref{lemma:cyc}, $\Delta(C_0^4)=\sum_{i=0}^3 (i,0)_4 C_i^4$, $\Delta(C_3^4)=\sum_{i=0}^3 (i,0)_4 C_{i+3}^4$, $\Delta(C_0^4,C_3^4)=\sum_{i=0}^3 (i,1)_4 C_{i+3}^4$ and $\Delta(C_3^4,C_0^4)=\sum_{i=0}^3 (i,3)_4 C_i^4$. So each element of $C_i^4$ ($0 \leq i \leq 3$) occurs in $\Delta(D)$ precisely
$$ (i,0)_4 + (i-3,0)_4 + (i-3,1)_4 + (i,3)_3$$
times.  We apply Lemma \ref{lemma:cyc} and Theorem \ref{thm:KaRaThm}.  Here $q \equiv 5 \mod 8$ so $f$ is odd. Using $(1,1)_4=(1,0)_4, (2,1)_4=(1,0)_4, (3,0)_4=(2,3)_4$ and $(3,3)_4=(3,0)_4$, we see that $\Delta(D)$ comprises:
\begin{itemize}
\item $(0,0)_4+(0,3)_4+2(1,0)_4=\frac{1}{16}(4q-12+4t)$ copies of $C_0^4$;
\item $2(1,0)_4+(1,3)_4+(2,0)_4=\frac{1}{16}(4q-12-4t)$ copies of $C_1^4$;
\item $(2,0)_4+ 2(2,3)_4+(3,1)_4=\frac{1}{16}(4q-12
+4t)$ copies of $C_2^4$;
\item $(0,0)_4+(0,1)_4+2(3,0)_4=\frac{1}{16}(4q-12-4t)$ copies of $C_3^4$.
\end{itemize}
So $\Delta(D)=\frac{1}{4}(q-3+t)(C_0^4+C_2^4) + \frac{1}{4}(q-3-t)(C_1^4+C_3^4)$, and hence $D$ is a skew PDS for the $(q,\frac{q-1}{2}, \frac{q-5}{4}, \frac{q-1}{4})$-PDS $C_0^2$ precisely if $t=-2$, and is a skew PDS for the $(q,\frac{q-1}{2}, \frac{q-5}{4}, \frac{q-1}{4})$-PDS $C_1^2$ precisely if $t=2$.
\end{proof}

The following result is an immediate consequence of Theorem \ref{thm:C0C3}.
\begin{theorem}\label{thm:C0C1}
Let $q$ be a prime power congruent to $5$ modulo $8$. Let $\alpha$ be a primitive element of $\mathbb{F}_q$. Let $q=s^2+t^2$ be the unique proper representation (with sign ambiguity resolved) as in Theorem \ref{thm:KaRaThm}. 

Then $C_0^4 \cup C_1^4$ is a Paley skew PDS corresponding to 
\begin{itemize}
    \item[(i)] the Paley $(q,\frac{q-1}{2}, \frac{q-5}{4}, \frac{q-1}{4})$-PDS  $C_1^2$, if $t=-2$; 
    \item[(ii)] the Paley $(q,\frac{q-1}{2}, \frac{q-5}{4}, \frac{q-1}{4})$-PDS  $C_0^2$, if $t=2$.
\end{itemize}
\end{theorem}
\begin{proof}
Let $E=C_0^4 \cup C_1^4$.  Since $E=\alpha D$, $\Delta(D)=\frac{1}{4}(q-3+t)(C_1^4+C_3^4) + \frac{1}{4}(q-3-t)(C_2^4+C_0^4)$.
\end{proof}

The negative and complement of this skew PDS also form skew PDSs.
\begin{theorem}
Let $q$ be a prime power congruent to $5$ modulo $8$. Let $\alpha$ be a primitive element of $\mathbb{F}_q$. Let $q=s^2+t^2$ be the unique proper representation (with sign ambiguity resolved) as in Theorem \ref{thm:KaRaThm}. Let $D=C_0^4 \cup C_3^4$.  Then
\begin{itemize}
\item[(i)] $G^* \setminus D=-D=C_1^4 \cup C_2^4$ is a Paley skew PDS corresponding to the Paley PDS $C_0^2$ if $t=-2$ and $C_1^2$ if $t=2$.
\item[(ii)] $G \setminus D = C_1^4 \cup C_2^4 \cup \{0\}$ is a skew PDS corresponding to the $(q,(q+1)/2,(q-1)/4,(q+3)/4)$-PDS $C_1^2 \cup \{0\}$ if $t=-2$ and $C_0^2 \cup \{0\}$ if $t=2$.
\end{itemize}
\end{theorem}
\begin{proof}
As $q \equiv 5 \mod 8$, $q-1=8k+4$ for some integer $k$.  So $-1=\alpha^{(q-1)/2}=\alpha^{4k+2} \in C_2^4$ and $-C_i^4=C_{i+2}^4$. \\
(i) Since $-D=\alpha^2 D$ and $\alpha^2 \in C_0^2$, we have that $\Delta(-D)=\Delta(\alpha^2 D)=\alpha^2 \Delta(D)=\frac{1}{4}(q-3+t)C_0^2 + \frac{1}{4}(q-3-t)C_3^2=\Delta(D)$ and so $-D$ is a skew PDS corresponding to the same PDS as $D$ in each case.  Observe that $-D=C_1^4 \cup C_2^4$ is the complement $G^* \setminus D= G^* \setminus (C_0^4 \cup C_3^4)$.\\
(ii) Here $\Delta(G \setminus D)=-(C_1^4 \cup C_2^4) \cup (C_1^4 \cup C_2^4) \cup \Delta(C_1^4 \cup C_2^4)=\mathbb{F}_q^* \cup \Delta(D)$ by (i).  This comprises $\frac{q+1+t}{4}$ copies of $C_0^2$ and $\frac{q+1-t}{4}$ copies of $C_1^3$.  Observe that $\Delta(C_i^2 \cup \{0\})=2C_i^2 \cup \Delta(C_i^2)$ and so $C_i^2 \cup \{0\}$ is a $(q,(q+1)/2, (q+3)/4,(q-1)/4)$-PDS.  So for $t=-2$, $G \setminus D$ is a skew PDS for $C_1 \cup \{0\}$ and for $t=2$, it is a skew PDS for $C_0 \cup \{0\}$, as required. This also follows from Theorem \ref{thm:skewPDScomp}.
\end{proof}

\begin{example}\label{ex:e4example}
Let $q=13$. Take primitive element $2$ of $\mathbb{F}_{13}$.  Then $13=9+4=s^2+t^2$ where $s=-3 \equiv 1 \mod 4$ and $t=\pm 2$, where the sign of $t$ is determined by $8=2^3 \equiv (-3)t^{-1} \mod 13$.   So $t^{-1} \equiv -7 \mod 13$, i.e., $t=-2$.  The cyclotomic classes are $C_0^4=\{1,3,9\}$, $C_1^4=\{2,5,6\}$, $C_2^4=\{4,10,12\}$ and $C_3^4=\{7,8,11\}$.  Hence $C_0^4 \cup C_3^4=\{1,3,7,8,9,11\}$ is a skew PDS corresponding to the Paley $(13,6,2,3)$-PDS $C_0^2=\{1,3,4,9,10,12\}$, and $C_0^4 \cup C_1^4=\{1,2,3,5,6,9\}$ is a skew PDS corresponding to the Paley $(13,6,2,3)$-PDS $C_1^2=\{2,5,6,7,8,11\}$.

Observe that if we take primitive element $7$ of $\mathbb{F}_{13}$, then $t=2$ (since $7^3=343 \equiv 18 \mod 13=(-3)t^{-1}$ so $t^{-1}=-6$)  and the cyclotomic classes are $C_0^4=\{1,3,9\}$, $C_1^4=\{7,8,11\}$, $C_2^4=\{4,10,12\}$ and $C_3^4=\{2,5,6\}$.  Hence $C_0^4 \cup C_3^4=\{1,2,3,5,6,9\}$ is a skew PDS corresponding to the PDS $C_1^2=\{2,5,6,7,8,11\}$, and $C_0^4 \cup C_1^4=\{1,3,7,8,9,11\}$ is a skew PDS corresponding to the PDS $C_0^2=\{1,3,4,9,10,12\}$.
\end{example}

\begin{example}
Table \ref{tab:skewPaleye=4} gives values of $q<10^4$ for which $t=\pm 2$ and the constructions of Theorem \ref{thm:C0C3} and Theorem \ref{thm:C0C1} give skew Paley PDSs corresponding to $C_0^2$ and $C_1^2$. All such $q$ satisfy $q \equiv 5 \mod 8$, but not every $q \equiv 5 \mod 8$ is of the necessary form.

\begin{table}[h]
	\centering
	\[
	\begin{array}{ccc}
		\hline
		q & s^2+t^2 & \text{Skew PDS parameters} \\
		\hline
		13 & (-3)^2+(\pm2)^2 & (13,6,2,3) \\
		29 & 5^2+(\pm2)^2 & (29,14,6,7) \\
		53 & (-7)^2+(\pm2)^2 & (53,26,12,13) \\
		125 & (-11)^2+(\pm2)^2 & (125,62,30,31) \\
		173 & 13^2+(\pm2)^2 & (173,86,42,43) \\
		229 & (-15)^2+(\pm2)^2 & (229,114,56,57) \\
		293 & 17^2+(\pm2)^2 & (293,146,72,73) \\
		733 & (-27)^2+(\pm2)^2 & (733,366,182,183) \\
		1093 & 33^2+(\pm2)^2 & (1093,546,272,273) \\
		1229 & (-35)^2+(\pm2)^2 & (1229,614,306,307) \\
		1373 & 37^2+(\pm2)^2 & (1373,686,342,343) \\
		2029 & 45^2+(\pm2)^2 & (2029,1014,506,507) \\
		2213 & (-47)^2+(\pm2)^2 & (2213,1106,552,553) \\
		3253 & 57^2+(\pm2)^2 & (3253,1626,812,813) \\
		4229 & 65^2+(\pm2)^2 & (4229,2114,1056,1057) \\
		4493 & (-67)^2+(\pm2)^2 & (4493,2246,1122,1123) \\
		5333 & 73^2+(\pm2)^2 & (5333,2666,1332,1333) \\
		7229 & 85^2+(\pm2)^2 & (7229,3614,1806,1807) \\
		7573 & (-87)^2+(\pm2)^2 & (7573,3786,1892,1893) \\
		9029 & (-95)^2+(\pm2)^2 & (9029,4514,2256,2257) \\
		9413 & 97^2+(\pm2)^2 & (9413,4706,2352,2353) \\
		\hline
	\end{array}
	\]
	\caption{Skew PDSs from Theorem \ref{thm:C0C3} and Theorem \ref{thm:C0C1}.}
	\label{tab:skewPaleye=4}
\end{table}

\end{example}

It can be proven that $C_0^4 \cup C_3^4$ is not a PDS nor a skew PDS for any $q \equiv 1 \mod 8$.

Note that a skew PDS corresponding to a Paley PDS gives an example of an almost difference set (see for example \cite{Ara}).
\begin{definition}
Let $G$ be a group of order $v$, and let $D$ be a subset of $G$ of size $k$. Then $D$ is a $(v,k,\lambda,t)$ \emph{almost difference set} (ADS) in $G$ if for some $t$-subset $T$ of $G^*$ the following multiset equation holds:
$$\Delta(D)=\lambda T^* + (\lambda+1) (G^* \setminus T).$$
\end{definition}

So every Paley PDS and corresponding skew PDS is a $(v,(v-1)/2, (v-5)/4,(v-1)/2)$-ADS possessing extra properties.  For prime $v$, the construction of Theorem \ref{thm:C0C1} is observed to form an ADS in Theorem 4 of \cite{DinHelLam}.  Theorems 7-9 of \cite{DinHelLam} present other constructions which are said to produce $(v,(v-1)/2, (v-5)/4,(v-1)/2)$-ADSs via unions of cyclotomic classes of order $6$ - these would be of potential interest for Paley skew PDSs. However, these seem to be incorrect: computational checking shows that they do not produce ADSs, and there are errors in the proof of Theorem 7.

\section{Cyclotomic skew PDSs using classes of order $8$}

In this section we use cyclotomic classes of order $8$ to obtain skew PDSs.  Necessary results on cyclotomic numbers of order $8$ may again be found in the Appendix.

We present a construction of a cyclotomic skew PDS which is not of Paley type. We will need the following useful result (Lemma 2 of \cite{DinPotWan}):
\begin{lemma}\label{lem:DinPotWan}
If $p \equiv 3 \mod 8$ and $m$ is even, then $2$ is a quartic residue in $\mathbb{F}_{p^m}$.
\end{lemma}

\begin{theorem}\label{thm:C_1^4skewPDS}
Let $q=p^m$ where $p \equiv 3 \mod 8$ and $m \equiv 2 \mod 4$.
Suppose that $q=x^2+4y^2=a^2+2b^2$
are the unique proper representations of $q$.
If $x+a=-2$, then $D=C_3^8 \cup C_5^8$ is a skew PDS corresponding to the $(q,(q-1)/4,(q-11-6x)/16,(q-3+2x)/16)$-PDS $C_0^4$.

Moreover, for $0 \leq i \leq 7$, the set $C_i^8 \cup C_{i+2}^8$ is a skew PDS corresponding to the $(q,(q-1)/4,(q-11-6x)/16,(q-3+2x)/16)$-PDS $C_{i+1}^4$.
\end{theorem}
\begin{proof}
Since $p \equiv 3 \mod 4$ and $q=p^{2r}$ (for some odd $r$), by Theorem \ref{thm:C_0^4PDS}, each $C_i^4$ is a PDS with the given parameters.  We show that $D=C_3^8 \cup C_5^8$ is a skew PDS corresponding to $C_0^4$.
By Theorem \ref{thm:e=8cyclo}, since $p \not\equiv 1 \mod 4$, the unique proper representation $q=x^2+4y^2$ has $x=\pm p^{m/2}$ and $y=0$.  Since $p \equiv 3 \mod 8$, the unique proper representation $q=a^2+2b^2$ has $a \equiv 1 \mod 4$.  By Lemma \ref{lem:DinPotWan}, since $p \equiv 3 \mod 8$ and $2r$ is even, $2$ is a quartic residue in $\mathbb{F}_q$.  Since $p^2 \equiv 9 \mod 16$ and $q=(p^2)^r$ where $r$ is odd, $q \equiv 9 \mod 16$. 
Now, $$\Delta(D)=\Delta(C_3^8)+\Delta(C_5^8)+\Delta(C_3^8,C_5^8)+\Delta(C_5^8,C_3^8).$$
Here $\Delta(C_3^8)=\sum_{i=0}^7 (i,0)_8 C_{i+3}^8$, $\Delta(C_5^8)=\sum_{i=0}^7 (i,0)_8 C_{i+5}^8$, $\Delta(C_3^8, C_5^8)=\sum_{i=0}^7 (i,6)_8 C_{i+5}^8$ and $\Delta(C_5^8,C_3^8)=\sum_{i=0}^7 (i,2)_8 C_{i+3}^8$.  From these, we see that each element of $C_i^8$ ($0 \leq i \leq 7$) occurs in $\Delta(D)$ precisely
$$N_i=(i-3,0)_8+(i-5,0)_8+(i-5,6)_8+(i-3,2)_8$$
times.  
Using Theorem \ref{thm:e=8cyclo}, we see that for $0 \leq i \leq 7$, $(i,0)_8=(i+4,0)_8$ and $(i,2)_8=(i+2,6)_8$, i.e., $(i-2,6)_8=(i+4,2)_8$. Hence 
$$N_j=N_{j+4}=(j+1,0)_8+(j-1,0)_8+(j+1,2)_8+(j-1,6)_8$$ 
for $0 \leq j \leq 3$, i.e., each element of $C_j^4$ occurs $N_j$ times.  From the formulae given in Theorem \ref{thm:e=8cyclo} for this case, the elements of each class occur with the following frequencies:
\begin{itemize}
\item $N_0=(q-3-2x+4a)/16$; 
\item $N_1=(q-7-2a+4y)/16$;
\item $N_2=(q-3+2x)/16$;
\item $N_3=(q-7-2a-4y)/16$.
\end{itemize}
Since $y=0$ and $x+a=-2$, we have $N_1=N_2=N_3=(q-3+2x)/16$. Thus the frequency of each element of $C_1^4 \cup C_2^4 \cup C_3^4$ equals $(q-3+2x)/16$ as required. We have that $N_0=(q-11-6x)/16$, which is distinct from the other frequencies since $x \neq -1$.
The second claim follows since $\Delta(\alpha^{-3}(C_3^8 \cup C_5^8))=\Delta(C_0^8 \cup C_2^8)=(q-11-6x)/16 C_1^4 + (q-3+2x)/16 (C_2^4+C_3^4+C_0^4)$ so $C_0^8 \cup C_2^8$ is a skew PDS corresponding to $\alpha^{-3}C_0^4=C_1^4$.  Hence $\Delta(C_i^8 \cup C_{i+2}^8)= \alpha^{i-3}\Delta(D)= ((q-11-6x)/16) C_{i+1}^4 + ((q-3+2x)/16) (C_i^4 \cup C_{i+2}^4 \cup C_{i+3}^4)$.
\end{proof}

The complement and negative of the above construction also yield skew PDSs:

\begin{theorem}
Let $q=p^m$ where $p \equiv 3 \mod 8$ and $m \equiv 2 \mod 4$.
Suppose that $q=x^2+4y^2=a^2+2b^2$ are the unique proper representations of $q$ with $x+a=-2$ and let $D=C_3^8 \cup C_5^8$. Then
\begin{itemize}
\item[(i)]
		$G \setminus D=C_0^8 \cup C_1^8 \cup C_2^8 \cup C_4^8 \cup C_6^8 \cup C_7^8 \cup \{0\}$ is a skew PDS corresponding to the $(q,(3q+1)/4,(9q + 2x + 5)/16,(9q-6x-3)/16)$-PDS $C_1^4 \cup C_2^4 \cup C_3^4 \cup \{0\}$.
\item[(ii)] $-D=C_1^8 \cup C_7^8$  is a skew PDS corresponding to the PDS $C_0^4$.
\end{itemize}
\end{theorem}
\begin{proof}
Part (i) is immediate by Theorem \ref{thm:skewPDScomp}. 
Let $\alpha$ be a primitive element of $\mathbb{F}_q$.  Since $q \equiv 9 \mod 16$, $-1 \in C_4^8$ and so $-C_i^8=C_{i+4}^8$; so $-D=\alpha^4 D=C_7^8 \cup C_1^8$; the fact that it is a skew PDS corresponding to $C_0^4$ follows from Theorem \ref{thm:C_1^4skewPDS}.
\end{proof}

We next provide a corollary of Theorem \ref{thm:C_1^4skewPDS}, with conditions which are easier to work with.  This formulation is inspired by a result of \cite{DinPotWan} on almost difference sets, which is mentioned in more detail later in this section.
\begin{corollary}\label{cor:C_1^4skewPDS}
Let $\ell \equiv 3 \mod 8$ be a prime power of the form $(c^2/2)+1$ for some $c \in \mathbb{Z}$. Let $q=\ell^2$.  
Then $D=C_3^8 \cup C_5^8$ is a skew PDS in $\mathbb{F}_q$ corresponding to the $(q,(q-1)/4,(q-11+6\ell)/16,(q-3-2\ell)/16)$-PDS $C_0^4$.
\end{corollary}
\begin{proof}
We show that the specified $q$ satisfies the conditions of Theorem \ref{thm:C_1^4skewPDS}. Here $\ell=p^r \equiv 3 \mod 8$; by considering the congruence classes modulo $8$ of the powers of an odd prime, we see that necessarily $p \equiv 3 \mod 8$ and $r$ is odd. So $q=\ell^2=p^{2r}$ satisfies the first condition.  Note that any suitable $c$ satisfies $c \equiv 2 \mod 4$. Since $\ell \not\equiv 1 \mod 4$ but $-\ell \equiv 1 \mod 4$,  the unique proper representation $q=x^2+4y^2$ has $x=-\ell$ and $y=0$. Now since $\ell=(c^2/2)+1$, $q=\ell^2=(\ell-2)^2+2c^2$. This is the unique proper representation $q=a^2+2b^2$, with $a=\ell-2 \equiv 1 \mod 4$ and $b=\pm c$.  So $a+(-\ell)=a+x=-2$.
\end{proof}

\begin{example}\label{ex:skewPDSC0C2}
\begin{itemize}
\item[(i)] For $q=9$ (where $l=(2^2/2)+1=3$) the two unique proper representations are $9=(-3)^2+0=1^2+2(\pm 2)^2$, i.e., $x=-3$, $y=0$, $a=1$ and $b=\pm 2$.  Since $x+a=-3+1=-2$, the construction applies here to show that $D=C_3^8 \cup C_5^8$ is a skew PDS for the $(9,2,1,0)$-PDS $C_0^4$. Here all $C_i^8$ are singleton sets.
Taking primitive element $\alpha \in \mathbb{F}_9$ to be a root of $x^2+x+2$, we see that $C_0^4=\{1,2\}$ and $D=C_3^8 \cup C_5^8=\{\alpha^3,\alpha^5\}=\{2\alpha+2,2\alpha\}$, with $\Delta(C_0^4)=\Delta(D)=\{1,2\}$.  
\item[(ii)] For $q=361$ (where $l=(6^2/2)+1=19$), the two unique proper representations are $361=(-19)^2+4 \cdot 0^2=17^2+2(\pm 6)^2$, i.e., $x=-19$, $y=0$, $a=17$ and $b=\pm 6$.  Since $x+a=-19+17=-2$, the construction applies to show that $C_3^8 \cup C_5^8$ is a skew PDS for the $(361,90,29,20)$-PDS $C_0^4$.

\item[(iii)] For $q=26569$ (where $l=(18^2/2)+1=163$), the two proper representations are $q=(-163)^2+0^2=161^2+2(\pm 18)^2$, i.e.,  $x=-163$, $y=0$, $a=161$ and $b=\pm 18$.  Here $D=C_3^8 \cup C_5^8$ forms a $(26569,6642,1721,1640)$ skew PDS corresponding  to the $(26569,6642,1721,1640)$-PDS $C_0^4$.

\item[(iv)] For $q=59049$ (where $l=(22^2/2)+1=243=3^5$), the two proper representations are $q=(-243)^2+0^2=241^2+2(\pm 22)^2$, i.e.,  $x=-243$, $y=0$, $a=241$ and $b=\pm 22$.  Here $D=C_3^8 \cup C_5^8$ forms a $(59049,14762,3781,3660)$ skew PDS corresponding to the $(59049,14762,3781,3660)$-PDS $C_0^4$.
\end{itemize}
\end{example}

Next, we present a cyclotomic construction for skew Paley PDSs using classes of order $8$.  In fact, a version of this construction has been shown to yield an almost difference set, in Theorem 1 of \cite{DinPotWan}; its proof implicitly contains much of the information needed to establish our result.  We note the similarities to the proof approach of Theorem \ref{thm:C_1^4skewPDS}.

\begin{theorem}\label{thm:e=8skewPaley}
Let $q=p^m$ where $p \equiv 3 \mod 8$ and $m \equiv 2 \mod 4$.
Suppose that $q=x^2+4y^2=a^2+2b^2$
are the unique proper representations of $q$.
If $a=x+4$, then $D=C_0^8 \cup C_1^8 \cup C_2^8 \cup C_5^8$ is a skew Paley $(q,(q-1)/2,(q-5)/4,(q-1)/4)$-PDS corresponding to $C_0^2$.
\end{theorem}
\begin{proof}
By Theorem \ref{thm:e=8cyclo}, since $p \not\equiv 1 \mod 4$, the unique proper representation $q=x^2+4y^2$ has $x=\pm p^{m/2}$ and $y=0$.  Since $p \equiv 3 \mod 8$, the unique proper representation $q=a^2+2b^2$ has $a \equiv 1 \mod 4$.  By Lemma \ref{lem:DinPotWan}, $2$ is a quartic residue in $\mathbb{F}_q$.

Let $D=C_0^8 \cup C_1^8 \cup C_2^8 \cup C_5^8$.  By analogous calculations to those in the proof of Theorem \ref{thm:C_1^4skewPDS}, and as performed in the proof of Theorem 1 of \cite{DinPotWan}, it can be established that the elements of $C_i^8$ each occur $N_i$  times, where $N_i=N_{i+4}$ ($0 \leq i \leq 3$) and:
\begin{itemize}
\item $N_0=(16q-48+8x-8a-48y)/64$;
\item $N_1=(16q-80-16x+16a-32y)/64$;
\item $N_2=(16q-48+8x-8a-16y)/64$;
\item $N_3=(16q-16)/64$.
\end{itemize}
Since $y=0$ and $a=x+4$, clearly frequencies $N_0$ and $N_2$ agree and equal $(16q-48+8x-8a)/64=(16q-80)/64=(q-5)/4$. So each element of $C_0^2=C_0^8 \cup C_2^8 \cup C_4^8 \cup C_6^8$ occurs $(q-5)/4$ times in $\Delta(D)$. Further,  $16q-80-16x+16a-32y=16q-16$, so $N_1$ equals $N_3$, and this value equals $(q-1)/4$.  So each element of $C_1^2=C_1^8 \cup C_3^8 \cup C_5^8 \cup C_7^8$ occurs $(q-1)/4$ times in $\Delta(D)$.
\end{proof}

The complement and negative of the above construction also yield skew PDSs:
\begin{theorem}
Let \(q=p^m\), where \(p\equiv 3 \pmod 8\) and \(m\equiv 2 \pmod 4\). Suppose that $q=x^2+4y^2=a^2+2b^2$ are the unique proper representations of \(q\) with \(a=x+4\), and let
$D=C_0^8\cup C_1^8\cup C_2^8\cup C_5^8$. 
	Then 
	\begin{itemize}
		\item[i)]
		$
		G\setminus D
		=
		C_3^8\cup C_4^8\cup C_6^8\cup C_7^8\cup\{0\}
		$
		is a skew PDS corresponding to the
	$	
		(q,(q+1)/2,(q+3)/4,(q-1)/4)$-PDS 
$ 		C_1^2\cup\{0\}.
		$
		\item[ii)]
		$
		-D=C_1^8\cup C_4^8\cup C_5^8\cup C_6^8
		$ 
		is a skew PDS corresponding to the Paley PDS \(C_0^2\).
	\end{itemize}
\end{theorem}

The following corollary uses the same (easier-to-use) condition as for the ADS in \cite{DinPotWan}.

\begin{corollary}\label{cor:C0C1C2C5easier}
Let $q=\ell^2$ where $\ell=d^2+2 \equiv 3 \mod 8$ is a prime power and $d \in \mathbb{Z}$.  Then the set $D=C_0^8 \cup C_1^8 \cup C_2^8 \cup C_5^8$ is a skew Paley PDS corresponding to $C_0^2$.
\end{corollary}
\begin{proof}
We show that the given $q$ satisfies the conditions of Theorem \ref{thm:e=8skewPaley}.  We take the equivalent condition: $q=\ell^2$ where $\ell=(c/2)^2+2 \equiv 3 \mod 8$ is a prime power, $c \in \mathbb{Z}$.  Clearly the first condition holds.
Since $\ell \not\equiv 1 \mod 4$ but $-\ell \equiv 1 \mod 4$,  the unique proper representation $q=x^2+4y^2$ has $x=-\ell$ and $y=0$. Now since $\ell=(c^2/4)+1$, $q=\ell^2=(-\ell+4)^2+2c^2$. This is the unique proper representation $q=a^2+2b^2$, with $a=-\ell+4 \equiv 1 \mod 4$ and $b=\pm c$.  So $x+4=(-\ell)+4=a$.
\end{proof}

\begin{example}
We give an explicit example of the construction of Theorem \ref{thm:e=8skewPaley} with $q=9$.  Here $9=(-3)^2+0=1^2+2(\pm 2)^2$ (also $q=3^2$ where $3=1^2+2$).  Let $\alpha$ be the root of the primitive polynomial $x^2+2x+2 \in \mathbb{F}_3[x]$, so $\alpha^2=\alpha+1$.  The cyclotomic classes of order $8$ are singleton sets, with $C_0^8=\{1\}, C_1^8=\{\alpha\}, C_2^8=\{\alpha+1\}$ and $C_0^5=\{2\alpha\}$.  It can be verified directly that for $D=\{1,\alpha,\alpha+1,2\alpha\}$, the multiset $\Delta(D)$ comprises $2$ copies of $\{\alpha,\alpha+2,2\alpha,2\alpha+1\}=C_1^2$ and $1$ copy of $\{1,2\alpha+1,2\alpha+2\}=C_0^2$. 
\end{example}

\begin{example}
Table \ref{tab:skewC14} gives values of $q<10^8$ for which the construction of Corollary \ref{cor:C0C1C2C5easier} yields a skew Paley PDS corresponding to $C_0^2$.

\begin{table}[h]
	\centering
	\[
	\begin{array}{ccc}
		\hline
		q=\ell^2 & \ell=d^2+2 & \text{Skew PDS parameters} \\
		\hline
		9 & 3=1^2+2 & (9,4,1,2) \\
		121 & 11=3^2+2 & (121,60,29,30) \\
		729 & 27=5^2+2 & (729,364,181,182) \\
		6889 & 83=9^2+2 & (6889,3444,1721,1722) \\
		51529 & 227=15^2+2 & (51529,25764,12881,12882) \\
		196249 & 443=21^2+2 & (196249,98124,49061,49062) \\
		1190281 & 1091=33^2+2 & (1190281,595140,297569,297570) \\
		2319529 & 1523=39^2+2 & (2319529,1159764,579881,579882) \\
		4108729 & 2027=45^2+2 & (4108729,2054364,1027181,1027182) \\
		10569001 & 3251=57^2+2 & (10569001,5284500,2642249,2642250) \\
		43072969 & 6563=81^2+2 & (43072969,21536484,10768241,10768242) \\
		96098809 & 9803=99^2+2 & (96098809,48049404,24024701,24024702) \\
		\hline
	\end{array}
	\]
	\caption{Skew PDSs from Corollary \ref{cor:C0C1C2C5easier}.}
	\label{tab:skewC14}
\end{table} 
\end{example}

\section{Disjoint and external partial difference families}

In \cite{HucJoh}, the following structures were introduced, which simultaneously generalise PDSs, disjoint difference families and external difference families.  For information on these difference families, see for example \cite{Bur,PatSti}.

\begin{definition}\label{DPDF}
Let $\mathcal D$ be a collection of $m$ disjoint $k$-subsets $\{ D_1, \ldots, D_m \}$ of $G^*$ and let $S=\cup_{i=1}^m D_i$.  Then
\begin{itemize}
\item[(i)]  $\mathcal D$ is an $(n,m, k, \lambda, \mu)$ Disjoint Partial Difference Family (DPDF) of $G$ if the following multiset equation holds: 
$$ \cup_{i=1}^m \Delta(D_i)= \lambda S + \mu (G^* \setminus S).$$
If $\lambda=\mu$, this is an $(n,m,k,\lambda)$ Disjoint Difference Family (DDF). If $m=1$ this is a $(n,k,\lambda,\mu)$ Partial Difference Set.
\item[(ii)] $\mathcal D$ is an $(n,m, k, \lambda, \mu)$ External Partial Difference Family (EPDF) of $G$ if the following multiset equation holds: 
$$ \cup_{i,j: i \neq j} \Delta(D_i, D_j)= \lambda S + \mu (G^* \setminus S).$$
If $\lambda=\mu$, this is an $(n,m,k,\lambda)$ External Difference Family (EDF).
\end{itemize}
\end{definition}
Various DPDF/EPDF constructions are known, including partitioning a cyclotomic class which itself forms a PDS or DS into smaller cyclotomic classes \cite{HucJoh}.

In the ``Further Work" section of \cite{HucJoh}, the following definition was presented, and it was observed that construction methods for these would be of interest.

\begin{definition}
Let $G$ be a group of order $n$ and let $T$ be a subset of $G^*$. 
\begin{itemize}
\item[(i)] A collection of $m$ disjoint subsets $\mathcal{D} = \{ D_1, \ldots, D_m \}$ in $G^*$, where each $|D_i|=k_i$, forms an $(n,m, k_1, \ldots, k_m; \lambda, \mu)$ Disjoint Partial Difference Family of $G$ (relative to $T$) if the following multiset equation holds: 
$$ \cup_i \Delta(D_i)= \lambda(T) + \mu (G^* \setminus T). $$
If all $|k_i|=k$, this is an $(n,m,k; \lambda, \mu)$-DPDF relative to $T$.
\item[(ii)] A collection of disjoint $m$ subsets $\mathcal{D} = \{ D_1, \ldots, D_m \}$ in $G^*$, where each $|D_i|=k_i$, forms an $(n,m, k_1, \ldots, k_m; \lambda, \mu)$ External Partial Difference Family (relative to $T$) if the following multiset equation holds: 
$$ \cup_{i \neq j} \Delta(D_i, D_j)= \lambda(T) + \mu (G^* \setminus T).$$
If all $|k_i|=k$, this is an $(n,m,k; \lambda, \mu)$-EPDF relative to $T$.
\end{itemize}
\end{definition}
We view the case of a standard DPDF/EPDF partitioning $T$ or its complement as a trivial case of a DPDF/EPDF relative to $T$. 

\subsection{DPDFs and EPDFs relative to Paley PDSs}

In this section, we obtain non-trivial constructions of DPDFs and EPDFs relative to Paley PDSs. We show that these can be built from cyclotomic classes and their unions; some constructions arise by directly using skew PDSs, while others take a different approach.  Any construction relative to $C_0^2$ is relative to $C_1^2$ and vice versa, but will be stated as relative to $C_0^2$ for consistency. 

For a collection of disjoint $k$-subsets $\mathcal{D}=\{D_1,\ldots, D_m\}$ of $G^*$, we will denote by $\mathrm{Int}(\mathcal{D})$ the multiset $\cup_i \Delta(D_i)$, and by $\mathrm{Ext}(\mathcal{D})$ the multiset $\cup_{i \neq j} \Delta(D_i,D_j)$.  We have the following multiset equation:
$$ \mathrm{Int}(\mathcal{D}) + \mathrm{Ext}(\mathcal{D})=\Delta(S)$$
where $S=\cup_{i=1}^m D_i$.

\subsubsection{Relative DPDFs/EPDFs from skew PDSs}
A skew PDS corresponding to PDS $A$ is clearly a one-set DPDF relative to $A$.  Unlike in the skew PDS setting (Proposition \ref{prop:noC_iskewPDS}), it is possible for a single cyclotomic class to be a one-set DPDF relative to a different class, since here the set-sizes can be different.

\begin{example}
Let $q \equiv 9 \mod 16$ and let $q=x^2+4y^2=a^2+2b^2$ be the unique proper representations of \(q\). If $2$ is a quartic residue and $a=1$, then $(2i,0)_8=(q-1-15)/2x$ and $(2i+1,0)_8=(q-3+2x)/64$ ($0 \leq i \leq 3$), and so $C_0^8$ is a $(q,1,(q-1)/8;(q-15-2x)/64,(q-3+2x)/64)$-DPDF relative to $C_0^2$. For example, in $\mathbb{F}_{1801}$, $C_0^8$ is a $(1801,1,225,29,27)$-DPDF relative to $C_0^2$.
\end{example}

We have the following general result concerning families of sets.
\begin{proposition}\label{prop:generalPDSswap}
Let $G$ be an abelian group of order $v$ and let $\{A_1,\ldots, A_m\}$ be a family of pairwise disjoint subsets of $G^*$.
For each $i$, let $D_i\subseteq G^*$ be a subset of size $k_i$ satisfying
	\[
	\Delta(D_i)
	=
	\lambda_i A_i^*+\mu_i(G^*\setminus A_i).
	\]
	If
	$
	\lambda_i-\mu_i=\delta
	$
	for all $i$, and all $D_i$ are pairwise disjoint, then
	$
	\mathcal D=\{D_1,\ldots,D_m\}
	$ 
	is a
	$ 
	(v,m,k_1,\ldots,k_m;
	\delta+\sum_{i=1}^m\mu_i,\sum_{i=1}^m\mu_i)$-DPDF 
	relative to
	$\cup_{i=1}^m A_i.$ 
	
\end{proposition}

\begin{proof}
	Let
	$
	A=\cup_{i=1}^m A_i.
	$	Since the sets $A_1,\ldots,A_m$ are pairwise disjoint, every element of
	$A$ belongs to exactly one of them. Let $u\in A$, and suppose that $u\in A_j$. Then $u$ occurs with
	multiplicity $\lambda_j$ in $\Delta(D_j)$, and with multiplicity
	$\mu_i$ in $\Delta(D_i)$ for each $i\neq j$. Hence $u$ occurs in
	$
	\mathrm{Int}(\mathcal D)=\sum_{i=1}^m\Delta(D_i)
	$
	precisely
	$
	\lambda_j+\sum_{i\neq j}\mu_i
	=
	(\lambda_j-\mu_j)+\sum_{i=1}^m\mu_i
	=
	\delta+\sum_{i=1}^m\mu_i
	$
	times.  
	Now let $u\in G^*\setminus A$. Then $u\notin A_i$ for every $i$, so
	$u$ occurs with multiplicity $\mu_i$ in each $\Delta(D_i)$. Therefore
	$u$ occurs in $\mathrm{Int}(\mathcal D)$ precisely
	$
	\sum_{i=1}^m\mu_i
	$ 
	times. 
\end{proof}

In \cite{HucJoh}, for $q \equiv 1 \mod 4$, DPDFs are built by partitioning the PDS $C_0^2$ into smaller ``even" cyclotomic classes $\{C_{2i}^e: 0\leq i<e/2\}$ for even $e$ (or $C_1^2$ into ``odd" classes); in some cases, the individual classes are PDSs.  As an illustration of Proposition \ref{prop:generalPDSswap}, we can take this set-up and ``swap" PDSs with skew PDSs to obtain DPDFs relative to $C_0^2$.

\begin{theorem}\label{thm:DPDF0237}
Let $q=p^m$, where $p\equiv 3\pmod 8$ is a prime and $m\equiv 2\pmod 4$.
Suppose that $q=x^2+4y^2=a^2+2b^2$ are the proper representations of $q$ and suppose $x+a=-2$.  Then
\begin{itemize}
\item[(i)]  $\mathcal D_1=\{C_3^8\cup C_5^8,\ C_2^8\cup C_6^8\}$ is a $(q,2,(q-1)/4,(q-7-2x)/8),(q-3+2x)/8)$-DPDF
	relative to $C_0^2$.
\item[(ii)] $\mathcal D_2=\{C_0^8\cup C_2^8,\ C_3^8\cup C_7^8\}$ is a $(q,2,(q-1)/4,(q-3+2x)/8,(q-7-2x)/8)$-DPDF
	relative to $C_0^2$.
\end{itemize}
\end{theorem}
\begin{proof}
In this setting, by Theorem \ref{thm:C_1^4skewPDS}, $C_3^8 \cup C_5^8$ is a skew PDS corresponding to the $(q,(q-1)/4,(q-11-6x)/16,(q-3+2x)/16)$-PDS $C_0^4$. So $\Delta(C_3^8 \cup C_5^8)$ comprises $(q-11-6x)/16$ copies of $C_0^4$ and $(q-3+2x)/16$ copies of $C_1^4 \cup C_2^4 \cup C_3^4$. Moreover, $C_2^8 \cup C_6^8=C_2^4$; by Theorem \ref{thm:C_0^4PDS}, $C_2^4$ is a  $(q,(q-1)/4,(q-11-6x)/16,(q-3+2x)/16)$-PDS. So $\Delta(C_2^8 \cup C_6^8)$ comprises $(q-11-6x)/16$ copies of $C_2^4$ and $(q-3+2x)/16$ copies of $C_0^4 \cup C_1^4 \cup C_3^4$.  Hence $\mathrm{Int}(\mathcal{D})$ comprises $(q-11-6x)/16 + (q-3+2x)/16=(q-7-2x)/8$ copies of $C_0^4 \cup C_2^4=C_0^2$, and $2(q-3+2x)/16=(q-3+2x)/8$ copies
of $C_1^4 \cup C_3^4=C_1^2$.
The proof for (ii) is similar.
\end{proof}

Next, we consider a different generalisation of the partitioning approach, by considering how relative DPDFs and EPDFs can partition skew PDSs.
\begin{proposition}
Let $T \subseteq G^*$ be an $(n,mk,\lambda,\mu)$-PDS.  Let $\mathcal{D}=\{D_1,\ldots,D_m\}$ be a collection of disjoint $k$-subsets of $G^*$.  If $\mathcal{D}$ is 
\begin{itemize}
\item[(i)] a DPDF relative to $T$ and an EPDF relative to $T$; or
\item[(ii)] a DPDF relative to $T$ and an EDF; or
\item[(iii)] a DDF and an EPDF relative to $T$
\end{itemize}    
then $\cup_i D_i$ is a skew PDS relative to $T$.
\end{proposition}
\begin{proof}
Let $S=\cup_{i=1}^m D_i$, so $|S|=mk=|T|$.   Then in each case (i)-(iii), there are non-negative integers $\lambda_1,\lambda_2,\mu_1,\mu_2$ such that 
$\mathrm{Int}(\mathcal{D})=\lambda_1 T + \mu_1 (G^* \setminus T) $ and $\mathrm{Ext}(\mathcal{D})=\lambda_2 T + \mu_2 (G^* \setminus T)$. Hence $\Delta(S)=(\lambda_1+\lambda_2)T +(\mu_1+\mu_2)(G^* \setminus T)$.
\end{proof}

So suitable DPDFs and EPDFs could be used to construct skew PDSs.  Conversely, under certain conditions, a given skew PDS may be partitioned into sets to obtain relative DPDFs and EPDFs; this is the direction on which we focus.

In our first construction, we take the skew PDSs of Theorems \ref{thm:C0C3} and \ref{thm:C0C1} (relative to $C_0^2$ and $C_1^2$), and take each class as an individual set of the DPDF/EPDF.  

\begin{theorem}
Let $q$ be a prime power congruent to $5$ modulo $8$. Let $\alpha$ be a primitive element of $\mathbb{F}_q$. Let $q=s^2+t^2$ be the unique proper representation (with sign ambiguity resolved) as in Theorem \ref{thm:KaRaThm}. Then $\mathcal{D}=\{C_0^4,C_3^4\}$ is 
\begin{itemize}
\item[(i)] a $(q,2,(q-1)/4; (q-5)/8,(q+3)/5)$-EPDF relative to $C_0^2$ and a $(q,2,(q-1)/4,(q-5)/8)$-DDF, if $t=-2$;
\item[(ii)] a $(q,2,(q-1)/4;(q+3)/5, (q-5)/8)$-EPDF relative to $C_0^2$ and a $(q,2,(q-1)/4,(q-5)/8)$-DDF, if $t=2$.
\end{itemize}
\end{theorem}
\begin{proof}
From the proof of Theorem \ref{thm:C0C3}, $\Delta(C_0^4)=((q-7+2s)/16)C_0^2 + ((q-3-2s)/16)C_1^2$ while $\Delta(C_3^4)=((q-7+2s)/16)C_1^2 + ((q-3-2s)/16)C_0^2$.  Thus $\mathrm{Int}(\mathcal{D})$ is $((q-5)/8) C_0^2 + ((q-5)/8) C_1^2=((q-5)/8) \mathbb{F}_q^*$.  So $\mathcal{D}$ is a DDF, with the parameters specified.  The properties of $\mathrm{Ext}(\mathcal{D})$ follow from the skew PDS property of $S=C_0^4 \cup C_3^4$ established in Theorem \ref{thm:C0C3}, and the fact that $\Delta(S)=\mathrm{Int}(\mathcal{D}) + \mathrm{Ext}(\mathcal{D})$.
\end{proof}

It is not always the case that taking cyclotomic classes whose union forms a skew PDS corresponding to a PDS $T$ will yield a DPDF or EPDF relative to $T$ (see Example \ref{ex:C0C2}).

The next result constructs DPDFs and EPDFs relative to $C_0^2$ and $C_1^2$ by partitioning the Paley skew PDSs of Theorems \ref{thm:C0C1} and \ref{thm:C0C3} in a different way, using a technique originally used to obtain EDFs in \cite{ChaDin}. 

We first recall the following well-known lemma (see for example \cite{Sto}).
\begin{lemma}\label{lem:2is}
Let $q$ be an odd prime power.  Then 
\begin{itemize}
\item[(i)] if $q \equiv 1,7 \mod 8$ then $2 \in C_0^2$;
\item[(ii)] if $q \equiv 3,5 \mod 8$ then $2 \in C_1^2$.
\end{itemize}
\end{lemma}

\begin{theorem}\label{thm:DPDFnotclasses}
Let $q$ be a prime power congruent to $5$ modulo $8$. Let $\alpha$ be a primitive element of $\mathbb{F}_q$. Let $q=s^2+t^2$ be the unique proper representation (with sign ambiguity resolved) as in Theorem \ref{thm:KaRaThm}. 
Let $D_i=\{i,2i\}$ ($i \in \mathbb{F}_q^*$) and let $\mathcal{D}=\{D_i: i \in C_0^4\}$.  Then 
\begin{itemize}
\item[(i)] if $2 \in C_1^4$, $\mathcal{D}$ is
\begin{itemize}
    \item a $(q,(q-1)/4,2,(q-5)/4)$-EDF and a $(q,(q-1)/4,2;1,0)$-DPDF relative to $C_0^2$ if $t=-2$;
    \item a $(q,(q-1)/4,2;(q-9)/4,(q-1)/4)$-EPDF and a $(q,(q-1)/4,2;1,0)$-DPDF relative to $C_0^2$ if $t=2$.
\end{itemize}
\item[(ii)] if $2 \in C_3^4$, $\mathcal{D}$ is
\begin{itemize}
    \item a $(q,(q-1)/4,2,(q-5)/4)$-EDF and a $(q,(q-1)/4,2;1,0)$-DPDF relative to $C_0^2$ if $t=2$;
    \item a $(q,(q-1)/4,2;(q-9)/4,(q-1)/4)$-EPDF and a $(q,(q-1)/4,2;1,0)$-DPDF relative to $C_0^2$ if $t=-2$.
\end{itemize}
\end{itemize}
\end{theorem}
\begin{proof}
Since $q \equiv 5 \mod 8$, $-1=\alpha^{(q-1)/2} \in C_2^4$, so $-C_i^4=C_{i+2}^4$.  Observe that the multiset $\mathrm{Int}(\mathcal{D})$ is 
$$\cup_{i \in C_0^4} \Delta(D_i)=\cup_{i \in C_0^4} \{i,-i\}=C_0^4 \cup -C_0^4=C_0^4 \cup C_2^4=C_0^2.$$

Since $q \equiv 5 \mod 8$, by Lemma \ref{lem:2is} we have $2 \in C_1^2$.  We consider two cases: $2 \in C_1^4$ and $2 \in C_3^4$.  In the first case, $\cup D_i= C_0^4 \cup 2C_0^4=C_0^4 \cup C_1^4$ while in the second case $\cup D_i= C_0^4 \cup 2C_0^4=C_0^4 \cup C_3^4$.  These are both Paley skew PDSs by Theorems \ref{thm:C0C1} and \ref{thm:C0C3}.

When $2 \in C_1^4$, $\cup_{i \in C_0^4} D_i=C_0^4 \cup C_1^4$. 
By Theorem \ref{thm:C0C1}, this is a skew PDS whose internal difference multiset comprises $(q-3+t)/4$ copies of $C_1^2$ and $(q-3-t)/4$ copies of $C_0^2$.  So $\mathrm{Ext}(\mathcal{D})$ comprises $(q-3+t)/4$ copies of $C_1^2$ and $(q-7-t)/4$ copies of $C_0^2$.  When $2 \in C_3^4$, $\cup_{i \in C_0^4} D_i=C_0^4 \cup C_3^4$. By Theorem \ref{thm:C0C3}, this is a skew PDS whose internal difference multiset comprises $(q-3+t)/4$ copies of $C_0^2$ and $(q-3-t)/4$ copies of $C_1^2$.  So $\mathrm{Ext}(\mathcal{D})$ comprises $(q-7+t)/4$ copies of $C_0^2$ and $(q-3-t)/4$ copies of $C_1^2$.  This completes the proof.
\end{proof}

	\begin{example}
		Let $q=13$. We first take primitive element $2$ of $\mathbb{F}_{13}$. By Example~\ref{ex:e4example}, we have
		$t=-2$, $C_0^4=\{1,3,9\}$, $C_1^4=\{2,5,6\}$, $C_2^4=\{4,10,12\}$, $C_3^4=\{7,8,11\}$,
		and hence $2\in C_1^4$. Let $D_i=\{i,2i\}$ for $i\in C_0^4$. Then 
		$\mathcal{D}= \{D_i : i \in C_0^4\}= \{D_1, D_3, D_9\}=\{\{1,2\},\{3,6\},\{9,5\}\}$ and 
		$\bigcup_{i\in C_0^4}D_i=\{1,2,3,5,6,9\}=C_0^4\cup C_1^4$.
		
		We determine the internal differences; we have 
		$\Delta(D_1)=\{1,12\}$, $\Delta(D_3)=\{3,10\}$, $\Delta(D_9)=\{4,9\}$. Thus
		$
		\cup_{i\in C_0^4}\Delta(D_i)=\{1,3,4,9,10,12\}=C_0^2,
		$
		so $\mathcal{D}$ is a $(13,3,2;1,0)$-DPDF relative to $C_0^2$.
		We now  determine the external differences. We have 
		$\Delta(D_1, D_3)=\{8,9,11,12\}$, 
		$\Delta(D_3, D_1)=\{1,2,5,4\}$, 
		$\Delta(D_1, D_9)=\{5,6,9,10\}$, 
		$\Delta(D_9, D_1)=\{3,4,7,8\}$, 
		$\Delta(D_3, D_9)=\{1,7,10,11\}$, 
		$\Delta(D_9, D_3)=\{2,3,6,12\}$. 
		Hence each element of $\mathbb{F}_{13}^*$ occurs exactly twice among the external differences, and therefore $\mathcal{D}$ is a $(13,3,2,2)$-EDF.
	
If  we take primitive element $7$ of $\mathbb{F}_{13}$, then again  by Example~\ref{ex:e4example}, we have
$t=2$, $C_0^4=\{1,3,9\}$, $C_1^4=\{7,8,11\}$, $C_2^4=\{4,10,12\}$, $C_3^4=\{2,5,6\}$, and hence $2\in C_3^4$.  By the above calculations, $\mathcal{D}$ forms a $(13,3,2,2)$-EDF and a $(13,3,2;1,0)$-DPDF.
\end{example}

\subsubsection{Relative DPDFs from cyclotomic classes and their unions}

In this section, we show that cyclotomic classes may be used to obtain DPDFs and EPDFs relative to a PDS $T$, in cases where the classes and their unions are not related to $T$.

We first present constructions using individual cyclotomic classes. 

\begin{theorem}\label{thm:DPDFC0C2}
Let $q=8f+1$ be a prime power with $f$ odd. Suppose that
$q=x^2+4y^2=a^2+2b^2$ are the unique proper representations of $q$. 
Then $\mathcal D=\{C_0^{8},\,C_2^{8}\}$ is:
\begin{itemize}
\item a $(q, 2, {(q-1)}/{8}; {(q-11-2x-4a)}/{32}, {(q-7+2x+4a)}/{32})$-DPDF relative to $C_0^2$, if $x + 2a \ne -1$;
\item a $(q, 2, (q-1)/8, (q-9)/32)$-DDF, if $x+2a=-1$. 
\end{itemize}
\end{theorem}

\begin{proof}
By Lemma \ref{lemma:cyc}, since $\Delta(C_j^8)=\sum_{i=0}^7 (i,0)_8 C_{i+j}^8$, each element of $C_i^{8}$ occurs precisely $(i,0)_8+(i-2,0)_8$ times in $\mathrm{Int}(\mathcal D)$. 
By Theorem \ref{thm:e=8cyclo}, $(0,0)_8=(4,0)_8, (1,0)_8=(5,0)_8, (2,0)_8=(6,0)_8$ and $(3,0)_8=(7,0)_8$.  So each element of $C_j^8$ occurs in $\mathrm{Int}(\mathcal{D})$ precisely $X=(0,0)_8+(2,0)_8$ times when $j \in \{0,2,4,6\}$, and precisely $Y=(1,0)_8+(3,0)_8$ times when $j \in \{1,3,5,7\}$.  For the cyclotomic numbers, by Theorem \ref{thm:e=8cyclo} there are two separate cases to consider, depending on whether or not $2$ is a quartic residue.  However, in both cases we obtain the same values for $X$ and $Y$, namely $X=(q-11-2x-4a)/32$ and $Y=(q-7+2x+4a)/32$. Thus $\mathcal D$ forms a relative DPDF with respect to $C_0^2$ whenever the two frequencies are distinct.
		If $x+2a=-1$, then the two frequencies coincide and every nonzero element of $\mathbb F_q$ occurs exactly
	${(q-9)}/{32}
	$ 
	times in the multiset of internal differences of $\mathcal D$, so $\mathcal D$ is a DDF. This completes the proof. 
\end{proof}

\begin{example}\label{ex:C0C2}
\begin{itemize}
\item[(i)] The condition of Corollary \ref {cor:C_1^4skewPDS} provides examples for Theorem \ref{thm:DPDFC0C2}.  If $\ell \equiv 3 \mod 8$ is a prime power of the form $(c^2/2)+1$ for some $c \in \mathbb{Z}$ and $q=\ell^2$ (so in particular $q=8f+1$, $f$ odd), then it has been shown that $x+a=-2$ (where $a=\ell-2$).  So in this setting, $x+2a \neq -1$ whenever $a \neq 1$, i.e.,  whenever $\ell \neq 3$, so $q \neq 9$. So all examples with $q>9$ given in Example \ref{ex:skewPDSC0C2} provide examples of DPDFs for Theorem \ref{thm:DPDFC0C2}.  We note that the union of the sets $C_0^8 \cup C_2^8$ is itself a skew PDS, but the PDS which it corresponds to is $C_1^4$ rather than $C_0^2$ (or $C_1^2$).
\item[(ii)] Examples for Theorem \ref{thm:DPDFC0C2} not covered by case (i) may be obtained by taking $q=p^{2m+1}$ with $p \equiv 9 \pmod{16}$. In this case, $q \equiv 9 \pmod{16}$ (hence $q=8f+1$ with $f$ odd), but since $p \equiv 1 \pmod{8}$, these values of $q$ fall outside the hypotheses of Corollary~\ref{cor:C_1^4skewPDS}. Both DPDFs and DDFs may be so obtained. For $q=89 = 5^2 + 4\cdot 4^2 = 9^2 + 2\cdot 2^2$, we have $x=5$ and $a=9$, and hence $x+2a=5+2\cdot 9=23 \neq -1$. Thus $\mathcal D=\{C_0^8,C_2^8\}$ gives a $(89,2,11;1,4)$-DPDF relative to $C_0^2$.  For $q=41=5^2+4 \cdot 2^2=(-3)^2+2 \cdot 4^2$, we have $x=5$ and $a=-3$ so $x+2a=-1$ and we obtain a $(41,2,5,1)$-DDF.
\end{itemize}
\end{example}

\begin{theorem}\label{thm:DPDF-single0146}
Let $q=8f+1$ be a prime power with $f$ odd. Suppose that
$q=x^2+4y^2=a^2+2b^2$ are the unique proper representations of $q$. Suppose that $2$ is a
	quartic residue. Let
	$
	\mathcal D=\{C_0^8,C_1^8,C_4^8,C_6^8\}.
	$
	If $a=1$, then $\mathcal D$ is a
	$
	(q,4,(q-1)/8;(q-12-x)/16,(q-6+x)/16)
	$-DPDF
	relative to $C_0^2$.
\end{theorem}

\begin{proof}
Here $q \equiv 9 \mod 16$ and $\mathrm{Int}(\mathcal D)
=\Delta(C_0^8)+\Delta(C_1^8)+\Delta(C_4^8)+\Delta(C_6^8)$.  Each element of $C_i^8$ occurs in $\mathrm{Int}(\mathcal D)$ precisely $N_i=(i,0)_8+(i-1,0)_8+(i-4,0)_8+(i-6,0)_8$ times. By Theorem~\ref{thm:e=8cyclo}, $N_0=N_4=2(0,0)_8+(7,0)_8+(2,0)_8$ and $N_2=N_6=(0,0)_8+(1,0)_8+2(2,0)_8$.  Since $2$ is a quartic residue,
$N_0=(q-11-x-a)/16$ and $N_2=(q-9-x-3a)/16$; using $a=1$, we get $N_0=N_2=(q-12-x)/16$.  Similarly,
$N_1=N_5=(0,0)_8+2(1,0)_8+(7,0)_8$ and $N_3=N_7=(2,0)_8+(1,0)_8+2(7,0)_8$. By Theorem~\ref{thm:e=8cyclo}, we obtain $N_1=(q-9+x+3a)/16$
and $N_3=(q-7+x+a)/16$; using $a=1$, we get $N_1=N_3=(q-6+x)/16$, as required.
\end{proof}

The sets of the relative DPDF may also be unions of cyclotomic classes.

\begin{theorem}\label{thm:DPDF2union1}
Let $q=8f+1$ be a prime power with $f$ odd. Suppose that
$q=x^2+4y^2=a^2+2b^2$
are the unique proper representations of $q$. Then $\mathcal D=\{C_0^8\cup C_1^8,\ C_2^8\cup C_3^8\}$ is:
\begin{itemize}
\item  a $(q,2, (q-1)/4;(q-5+2y-2b)/{8}, (q-5-2y+2b)/8)$-DPDF relative to $C_0^2$, if $y \ne b$;
\item a $(q,2,(q-1)/4,(q-5)/8)$-DDF,  if $y=b$.
\end{itemize}
\end{theorem}
\begin{proof}
	We have
	\[
	\mathrm{Int}(\mathcal D)
	=
	\Delta(C_0^8)+\Delta(C_1^8)
	+\Delta(C_1^8,C_0^8)+\Delta(C_0^8,C_1^8)
	+\Delta(C_2^8)+\Delta(C_3^8)
	+\Delta(C_3^8,C_2^8)+\Delta(C_2^8,C_3^8).\]
	 	By Lemma~\ref{lemma:cyc},
	$
	\Delta(C_j^8)=\sum_{i=0}^7 (i,0)_8C_{i+j}^8,
	$ and 	$\Delta(C_{j+l}^8,C_l^8)=\sum_{i=0}^7 (i,j)_8C_{i+l}^8.
	$ 
	Hence each element of $C_i^8$ occurs in $\mathrm{Int}(\mathcal D)$
	precisely
	\[
		(i,0)_8+(i-1,0)_8+(i,1)_8+(i-1,7)_8
		+(i-2,0)_8+(i-3,0)_8+(i-2,1)_8+(i-3,7)_8	\]
	times.
	By 
	Theorem~\ref{thm:e=8cyclo}, we have  
	$
	(7,0)_8=(3,0)_8=(1,1)_8,
	(6,0)_8=(2,0)_8, 
	(5,0)_8=(1,0)_8=(4,1)_8=(7,7)_8,
    (4,0)_8=(0,0)_8,
	(6,1)_8=(1,7)_8,
    (3,7)_8=(0,1)_8$
    and
    $
	(5,7)_8=(2,1)_8.
	$
	Therefore each element of $C_j^8$ occurs in $\mathrm{Int}(\mathcal D)$
	precisely
	$
	X=(0,0)_8+(1,1)_8+(2,0)_8+2(1,0)_8+(0,1)_8+(1,7)_8+(2,1)_8
	$ 
	times when $j\in\{0,2,4,6\}$.
	
	Similarly, using
	$ 
	(3,1)_8=(2,1)_8=(6,7)_8,
	(5,1)_8=(0,7)_8,
	(7,1)_8=(1,2)_8=(2,7)_8
    $
    and
    $
    (4,7)_8=(1,1)_8
	$ 
	each element of $C_j^8$ occurs in $\mathrm{Int}(\mathcal D)$ precisely
	$
	Y=(1,0)_8+(0,0)_8+2(1,1)_8+(2,0)_8+(0,7)_8+(1,2)_8+(2,1)_8
	$
	times when $j\in\{1,3,5,7\}$. 

For the cyclotomic numbers, by Theorem~\ref{thm:e=8cyclo} there are two
	separate cases to consider, depending on whether or not $2$ is a quartic
	 residue  in $\mathbb F_q$. However, in both cases we obtain the same values
	for $X$ and $Y$, namely
	$
	X=(q-5+2y-2b)/8,$ and $ 
	Y=(q-5-2y+2b)/8.
	$ Thus $\mathcal D$ forms a relative DPDF with respect to $C_0^2$ whenever
	the two frequencies are distinct. If $y=b$, then the two frequencies
	coincide and every nonzero element of $\mathbb F_q$ occurs exactly
	$
	(q-5)/8
	$ times in the multiset of internal differences of $\mathcal D$, so
	$\mathcal D$ is a DDF.
\end{proof}
\begin{theorem}\label{thm:DPDF2union23}
Let $q=8f+1$ be a prime power with $f$ odd. Suppose that
$q=x^2+4y^2=a^2+2b^2$ are the unique proper representations of $q$. 
    
Let
	$
	\mathcal D_1=\{C_0^8\cup C_3^8,\ C_1^8\cup C_6^8\}
	$
	and
	$
	\mathcal D_2=\{C_0^8\cup C_5^8,\ C_2^8\cup C_7^8\}.
	$
Then
\begin{itemize}
\item if $y\ne -b$, $\mathcal D_1$ is a
	$
	(q,2,(q-1)/4;(q-5-2y-2b)/8,(q-5+2y+2b)/8)
	$-DPDF relative to $C_0^2$, and $\mathcal D_2$ is a
	$
	(q,2,(q-1)/4;(q-5+2y+2b)/8,(q-5-2y-2b)/8)
	$-DPDF relative to $C_0^2$;
\item if $y=-b$, $\mathcal D_1$ and $\mathcal D_2$ are $(q,2,(q-1)/4,(q-5)/8)$-DDFs.
\end{itemize}
\end{theorem}

\begin{proof}
	We have
	\[
	\mathrm{Int}(\mathcal D_1)
	=
	\Delta(C_0^8)+\Delta(C_3^8)
	+\Delta(C_3^8,C_0^8)+\Delta(C_0^8,C_3^8)
	+\Delta(C_1^8)+\Delta(C_6^8)
	+\Delta(C_6^8,C_1^8)+\Delta(C_1^8,C_6^8).
	\]
	By Lemma~\ref{lemma:cyc}, each element of $C_i^8$ occurs in $\mathrm{Int}(\mathcal D_1)$
	precisely
	$(i,0)_8+(i-3,0)_8+(i,3)_8+(i-3,5)_8
	+(i-1,0)_8+(i-6,0)_8+(i-1,5)_8+(i-6,3)_8$ 
	times.
	By Theorem~\ref{thm:e=8cyclo}, each element of $C_j^8$ occurs in $\mathrm{Int}(\mathcal D_1)$
	precisely
	$
	X=(0,0)_8+(1,0)_8+(0,3)_8+2(1,1)_8+(2,0)_8+(1,3)_8+(1,7)_8
	$
	times when $j\in\{0,2,4,6\}$, and precisely
	$
	Y=(0,0)_8+2(1,0)_8+(1,1)_8+(2,0)_8+(0,5)_8+(1,3)_8+(1,2)_8
	$
	times when $j\in\{1,3,5,7\}$.
	Regardless of whether or not $2$ is a quartic 
	residue in $\mathbb F_q$, we obtain the same values
	for $X$ and $Y$, namely
	$
	X=(q-5-2y-2b)/8, 
	Y=(q-5+2y+2b)/8.
	$
Similarly, for $\mathcal D_2$, each element of $C_j^8$ occurs in $\mathrm{Int}(\mathcal D_2)$ precisely $Y=(q-5+2y+2b)/8$ times when $j\in\{0,2,4,6\}$, and precisely $X=(q-5-2y-2b)/8$ times when $j\in\{1,3,5,7\}$.
If $y=-b$, then in both cases the two frequencies coincide and equal $(q-5)/8$.
\end{proof}

We have the following corollary of Theorem~\ref{thm:DPDF2union23} for $q=p^2$, where $p\equiv5\pmod8$.

\begin{corollary}\label{cor:DPDFp2}
	Let $q=p^2$, where $p\equiv5\pmod8$ is a prime. Let
	$
	\mathcal D_1=\{C_0^8\cup C_3^8,\ C_1^8\cup C_6^8\}
	$ 
	and
	$
	\mathcal D_2=\{C_0^8\cup C_5^8,\ C_2^8\cup C_7^8\}.
	$
	Then $\mathcal D_1$ is a
	$
	(q,2,(q-1)/4;(q-5-2y)/8,(q-5+2y)/8)
	$- 
	DPDF relative to $C_0^2$, and $\mathcal D_2$ is a
	$
	(q,2,(q-1)/4;(q-5+2y)/8,(q-5-2y)/8)
	$-
	DPDF relative to $C_0^2$.
\end{corollary}

\begin{proof}
	Since $p\equiv5\pmod8$, we have $q=p^2\equiv9\pmod{16}$, so $q=8f+1$
	with $f$ odd. Hence Theorem~\ref{thm:DPDF2union23} applies. Since $p\equiv5\pmod8 \equiv 1 \pmod 4$, by Theorem \ref{thm:e=8cyclo} we have $b=0$ in the
	proper representation $q=a^2+2b^2$ and $y \neq 0$ in the proper representation $x^2+4y^2$. It follows that the frequencies in
	Theorem~\ref{thm:DPDF2union23} reduce to
	$
	(q-5-2y)/8$, and $ (q-5+2y)/8,
	$ 
	which completes the proof.
\end{proof}

\begin{example}
For $q=25=(-3)^2+4(\pm2)^2=5^2+2(0)^2$, we have $y \neq 0$ and $b=0$.  In $\mathbb{F}_{25}$, take primitive element $\alpha$ to be a root of $x^2+2x+3 \in \mathbb{F}_5[x]$. Then $C_0^8=\{1, \alpha+3, 4 \alpha+1\}$, $C_1^8=\{\alpha,\alpha+2,3\alpha+3\}$, $C_3^8=\{\alpha+1,2\alpha,2\alpha+4\}$, $C_6^8=\{3,2\alpha+3,3\alpha+4\}$. Then $\mathcal D_1$ of Corollary \ref{cor:DPDFp2} is a $(25,2,6,2,3)$-DPDF relative to $C_0^2$ (here $y=2$).
\end{example}

\begin{theorem}\label{thm:DPDF2union0127-01361}
	Let \(q=p^m\), where
	$
	p\equiv5\pmod8
	$
	and
	$
	m\equiv2\pmod4.
	$
	Suppose that $q=x^2+4y^2$
	is the unique proper representation of \(q\). Let
	$
	\mathcal D_1=\{C_0^8\cup C_1^8,\ C_2^8\cup C_7^8\}
	$
	and
	$
	\mathcal D_2=\{C_0^8\cup C_1^8,\ C_3^8\cup C_6^8\}.
	$
	Then \(\mathcal D_1\) and \(\mathcal D_2\) are
	$
	(q,2,(q-1)/4;(q-5+2y)/8,(q-5-2y)/8)
	$
	-DPDFs relative to \(C_0^2\).
\end{theorem}

\begin{proof}
	Since \(p\equiv5\pmod8\) and \(m\equiv2\pmod4\), by
	Theorem~\ref{thm:e=8cyclo}, the representation
	$
	q=a^2+2b^2
	$
	satisfies \(b=0\). We have
	\[
	\mathrm{Int}(\mathcal D_1)
	=\Delta(C_0^8)+\Delta(C_1^8)
	+\Delta(C_1^8,C_0^8)+\Delta(C_0^8,C_1^8)
	+\Delta(C_2^8)+\Delta(C_7^8)
	+\Delta(C_7^8,C_2^8)+\Delta(C_2^8,C_7^8).
	\]
	By Lemma~\ref{lemma:cyc}, each element of \(C_i^8\) occurs in \(\mathrm{Int}(\mathcal D_1)\)
	precisely
	\[
	N_i=(i,0)_8+(i-1,0)_8+(i,1)_8+(i-1,7)_8
	+(i-2,0)_8+(i-7,0)_8+(i-2,5)_8+(i-7,3)_8
	\]
	times. By Theorem~\ref{thm:e=8cyclo}, $N_0=N_4=(0,0)_8+2(1,0)_8+(1,1)_8+(2,0)_8
	+(0,1)_8+(1,2)_8+(1,3)_8
	$
	and 
	$
	N_2=N_6=(0,0)_8+2(1,0)_8+(1,1)_8+(2,0)_8
	+(2,1)_8+(1,7)_8+(0,5)_8
	$. In Theorem~\ref{thm:e=8cyclo} there are two
	separate cases to consider, depending on whether or not \(2\) is a quartic 
	residue  in \(\mathbb F_q\); both cases yield
	$
	N_0=(q-5+2y-2b)/8
	$
	and
	$
	N_2=(q-5+2y+2b)/8.
	$
	Since \(b=0\), we get
	$
	N_0=N_2=(q-5+2y)/8.
	$
	On the other hand,
	$
	N_1=N_5=(0,0)_8+(1,0)_8+2(1,1)_8+(2,0)_8
	+(0,7)_8+(1,3)_8+(1,7)_8
	$
	times
    and
    $N_3=N_7=(0,0)_8+(1,0)_8+2(1,1)_8+(2,0)_8
	+(3,1)_8+(2,7)_8+(0,3)_8
	$
	times. Again by Theorem~\ref{thm:e=8cyclo}, in both cases we obtain
	$
	N_1=(q-5-2y-2b)/8
	$
	and
	$
	N_3=(q-5-2y+2b)/8.
	$
	Since \(b=0\), we get
	$
	N_1=N_3=(q-5-2y)/8.
	$
The proof for $\mathcal{D}_2$ is analogous.
\end{proof}

\subsubsection{Relative EPDFs from cyclotomic classes}
In this section, we use collections of cyclotomic classes to obtain relative EPDFs.

\begin{theorem}\label{thm:EPDF2QR}
Let $q=8f+1$ be a prime power with $f$ even, and suppose that
$q=x^2+4y^2=a^2+2b^2$ are the unique proper representations of $q$.  Suppose that $2$ is a quartic residue.
Let $\mathcal D=\{C_0^{8},C_4^{8}\}$.  If $a=1$, then $\mathcal D$ is a
	$(q,2, (q-1)/{8};(q+1-2x)/32,(q-3+2x)/32)$-EPDF
	relative to $C_0^2$.
\end{theorem}
\begin{proof}
Since $\Delta(C_4^8,C_0^8)=\sum_{i=0}^7 (i,4)_8 C_i^8$ and $\Delta(C_0^8,C_4^8)=\sum_{i=0}^7 (i,4)_8 C_{i+4}^8$, each $C_i^8$ occurs $N_i=(i,4)_8+(i+4,4)_8$ times in $\mathrm{Ext}(\mathcal{D})$.  By Theorem \ref{thm:e=8cyclo} we have $(i,4)_8=(i+4,4)_8$ ($0 \leq i \leq 3$).  Hence $N_i=N_{i+4}=2(i,4)_8$ ($0 \leq i \leq 3$), i.e each $C_i^4$ occurs $2(i,4)_8$ times.

By Theorem \ref{thm:e=8cyclo}, as $2$ is a quartic residue, we obtain 
$(1,4)_8=(3,4)_8=(q+1+2x-4a)/64$, so each element of $C_1^2$ occurs $(q+1+2x-4a)/32$ times in $\mathrm{Ext}(\mathcal{D})$. We have $(0,4)_8=(q-7-2x+8a)/64$ and $(2,4)_8=(q+1-2x)/64$; these are equal if $a=1$, in which case $C_0^2$ occurs $(q+1-2x)/32$ times (and $C_1^2$ occurs $(q-3+2x)/32$ times).
\end{proof}

\begin{theorem}\label{thm:EPDF1111}
	Let $q=8f+1$ be a prime power with $f$ even, and suppose that $q=x^2+4y^2=a^2+2b^2$ are the unique proper representations of $q$. Suppose that $2$ is not a
	quartic residue. Then
	$
	 D=\{C_0^8,C_1^8,C_4^8,C_5^8\}
	$ 
	is:
	\begin{itemize}
		\item a
		$
		(q,4,(q-1)/8;(3q-3+8y)/16,(3q-3-8y)/16)
		$-EPDF relative to $C_0^2$, if $a=-3$ and $y\ne0$;
		
		\item a
		$
		(q,4,(q-1)/8,(3q-3)/16)
		$-EDF, if $a=-3$ and $y=0$.
	\end{itemize}
\end{theorem}

\begin{proof}
	We have
	\[
	\begin{aligned}
		\mathrm{Ext}(\mathcal D)
		={}&
		\Delta(C_0^8,C_1^8)+\Delta(C_1^8,C_0^8)
		+\Delta(C_0^8,C_4^8)+\Delta(C_4^8,C_0^8)\\
		&+\Delta(C_0^8,C_5^8)+\Delta(C_5^8,C_0^8)
		+\Delta(C_1^8,C_4^8)+\Delta(C_4^8,C_1^8)\\
		&+\Delta(C_1^8,C_5^8)+\Delta(C_5^8,C_1^8)
		+\Delta(C_4^8,C_5^8)+\Delta(C_5^8,C_4^8).
	\end{aligned}
	\]
	By Lemma~\ref{lemma:cyc},
	$
	\Delta(C_{j+l}^8,C_l^8)
	=
	\sum_{i=0}^7(i,j)_8C_{i+l}^8.
	$
	Hence each element of $C_i^8$ occurs in $\mathrm{Ext}(\mathcal D)$ precisely
	\[
	\begin{aligned}
		&N_i=(i-1,7)_8+(i,1)_8+(i-4,4)_8
		+(i,4)_8+(i-5,3)_8+(i,5)_8\\
		&\quad +(i-4,5)_8+(i-1,3)_8+(i-5,4)_8
		+(i-1,4)_8+(i-5,7)_8+(i-4,1)_8
	\end{aligned}
	\]
	times.
	
	By Theorem~\ref{thm:e=8cyclo}, we have
	$
	(0,1)_8=(7,7)_8,
	(0,4)_8=(4,4)_8,
	(0,5)_8=(3,3)_8,
	$
	$
	(3,7)_8=(4,1)_8=(4,5)_8=(7,3)_8,
	$
	and
	$
	(3,4)_8=(7,4)_8$.
	So $N_0=N_4=
	2(0,1)_8+2(0,4)_8+2(0,5)_8
	+4(3,7)_8+2(3,4)_8.
	$
	Similarly, using
	$
	(1,3)_8=(6,5)_8,
	(1,7)_8=(2,1)_8,
	$
	$
	(2,4)_8=(5,7)_8=(6,4)_8,
	(2,5)_8=(5,3)_8,
	$
	and
	$
	(1,6)_8=(6,1)_8,
	$
	we have
	$
	N_2=N_6=
	2(1,3)_8+2(3,7)_8+2(1,7)_8
	+2(2,4)_8+2(2,5)_8+2(1,6)_8
	$.
	For the cyclotomic numbers, by Theorem~\ref{thm:e=8cyclo}, since $2$
	is not a quartic residue, we obtain
	$
	N_0=(12q-36-8a+32y)/64
	$
	and
	$
	N_2=(12q+12+8a+32y)/64.
	$
	Thus, if $a=-3$, then
	$
	N_0=N_2=(3q-3+8y)/16.
	$
	
	On the other hand, using
	$
	(0,3)_8=(5,5)_8,
	(0,7)_8=(1,1)_8,
	$
	$
	(3,7)_8=(5,1)_8,
	$
	and
	$
	(3,4)_8=(4,3)_8=(4,7)_8=(1,5)_8$,
	we have
    $N_1=N_5=
	2(0,3)_8+2(0,4)_8+2(0,7)_8
	+2(3,7)_8+4(3,4)_8
	$.
	Using
	$
	(2,3)_8=(3,2)_8=(5,7)_8=(7,5)_8,
	(2,7)_8=(3,1)_8,
	$
	$
	(3,4)_8=(7,4)_8,
	(3,5)_8=(6,3)_8,
	$
	and
	$
	(6,7)_8=(7,1)_8$,
    we have 
	$N_3=N_7=
	2(2,3)_8+2(2,4)_8+2(2,7)_8
	+2(3,4)_8+2(3,5)_8+2(6,7)_8
	$.
	Again by Theorem~\ref{thm:e=8cyclo}, we obtain
	$
	N_1=(12q-36-8a-32y)/64
	$
	and
	$
	N_3=(12q+12+8a-32y)/64.
	$
	Thus, if $a=-3$, then
	$
	N_1=N_3=(3q-3-8y)/16$.
If $y=0$, then the two frequencies coincide and every nonzero element of $\mathbb F_q$ occurs exactly
$(3q-3)/16$ times in $\mathrm{Ext}(\mathcal D)$. 
\end{proof}

\begin{example}
Let $q=17=1^2+4(\pm2)^2=(-3)^2+2(\pm 2)^2$; here $a=-3$ and $y \neq 0$.  Take primitive element $\alpha=3$ of $\mathbb{F}_{17}$; then $\alpha^6=2 \in C_2^4$ and using  Theorem \ref{thm:EPDF1111}, $D=\{ \{1,16\}, \{3,14\}, \{4,13\}, \{5,12\}\}$ is a $(17,4,2; 4,2)$-EPDF relative to $C_0^2$.
\end{example}

We end this section with a construction in which the sets are of different sizes.
\begin{theorem}\label{thm:EPDF026}
Let $q=8f+1$ be a prime power with $f$ odd. Suppose that
$q=x^2+4y^2=a^2+2b^2$ are the unique proper representations of $q$. Let
	$
	\mathcal D=\{C_0^8,\ C_2^8\cup C_6^8\}.
	$ 
	Then $\mathcal D$ is:
	\begin{itemize}
		\item a
		$(
		q,2;(q-1)/8,(q-1)/4;
		(q-3+2x)/16,
		(q+1-2x)/16
		)$-EPDF
		relative to $C_0^2$, if $x\ne1$;
		
		\item a
		$
	(q,2;(q-1)/8,(q-1)/4;
		(q-1)/16)$-EDF,
	if $x=1$.
	\end{itemize}
\end{theorem}
\begin{proof}
	Here $q \equiv 9 \mod 16$. 
	We have
	\[
	\mathrm{Ext}(\mathcal D)
	=
	\Delta(C_0^8,C_2^8)+\Delta(C_2^8,C_0^8)
	+\Delta(C_0^8,C_6^8)+\Delta(C_6^8,C_0^8).
	\]
	By Lemma~\ref{lemma:cyc},
	$
	\Delta(C_{j+l}^8,C_l^8)
	=
	\sum_{i=0}^7(i,j)_8C_{i+l}^8.
	$
	Hence each element of $C_i^8$ occurs in $\mathrm{Ext}(\mathcal D)$ precisely
	$
	(i-2,6)_8+(i,2)_8+(i-6,2)_8+(i,6)_8
	$ 
	times.
	
	 By
	Theorem~\ref{thm:e=8cyclo}, we have
	$
	(6,6)_8=(2,2)_8=(2,0)_8,
	$
	$
	(0,2)_8=(2,6)_8,
	$
	and
	$
	(0,6)_8=(2,4)_8.
	$
	Therefore each element of $C_j^8$ with $j\in\{0,2,4,6\}$ occurs in
	$\mathrm{Ext}(\mathcal D)$ precisely
	$
	X=(0,2)_8+(0,6)_8+2(2,0)_8
	$ 
	times. Similarly, using
	$
	(7,6)_8=(1,7)_8,
	$
	$
	(1,2)_8=(3,6)_8,
	$
	and
	$
	(1,6)_8=(3,2)_8,
	$
	each element of $C_j^8$ with $j\in\{1,3,5,7\}$ occurs in
	$\mathrm{Ext}(\mathcal D)$ precisely
	$
	Y=(1,2)_8+(1,6)_8+(1,7)_8+(3,2)_8
	$ 
	times. 	For the cyclotomic numbers, by Theorem~\ref{thm:e=8cyclo} there are two
	separate cases to consider, depending on whether or not $2$ is a quartic
	residue in $\mathbb F_q$. In both cases we obtain
	$
	X=(q-3+2x)/16
	$
	and
	$
	Y=(q+1-2x)/16.
	$
	Therefore every element of $C_0^2$ occurs exactly
	$
	(q-3+2x)/16
	$
	times in $\mathrm{Ext}(\mathcal D)$, while every element of $C_1^2$
	occurs exactly
	$
	(q+1-2x)/16
	$
	times. Hence $\mathcal D$ is an EPDF relative to $C_0^2$ whenever
	$x\ne1$. If $x=1$, then the two frequencies coincide and every nonzero element of
	$\mathbb F_q$ occurs exactly
	$
	(q-1)/16
	$
	times in $\mathrm{Ext}(\mathcal D)$.
\end{proof}

\begin{example}
For $q=9$, we have $x=-3$ and $y=0$.  Taking primitive element $\alpha$ to be a root of $x^2+x+2$, applying Theorem \ref{thm:EPDF026} we have $C_0^8=\{1\}$, $C_2^8 \cup C_6^8=\{2\alpha+1,2\alpha\}$ and $\mathrm{Ext}(\mathcal{D})=\{\alpha,\alpha+1,2\alpha, 2\alpha+2\}=C_1^2$, so $\mathcal{D}$ is a $(9,2,1,2;0,1)$-EPDF relative to $C_0^2$.
\end{example}

\subsection{New DPDFs and EPDFs}
In this final section, we show that similar approaches to that of Theorem \ref{thm:DPDFnotclasses}, based on those of \cite{ChaDin}, will yield new constructions of standard (i.e., not relative) DPDFs and EPDFs, which partition $C_0^2$.  While previous cyclotomic DPDF/EPDF constructions have partitioned certain cyclotomic classes into smaller classes \cite{HucJoh}, the new constructions differ from these in that the smaller sets are not in general cyclotomic classes themselves.

\begin{theorem}\label{thm:q5mod8epdf}
	Let $q$ be a prime power congruent to $5$ modulo $8$.  
	Let $\gamma \in C_2^4$ and define $D_i=\{i,\gamma i\}$ ($i \in \mathbb{F}_q^*$). Let $\mathcal{D}=\{D_i: i \in C_0^4\}$. Then $\mathcal{D}$ is
	 a $(q,(q-1)/4,2,(q-5)/4)$-EDF and a $(q,(q-1)/4,2;0,1)$-DPDF if $1-\gamma \notin C_0^2$;
		\end{theorem}

\begin{proof}
	Since $q \equiv 5 \mod 8$, $-1=\alpha^{(q-1)/2} \in C_2^4$, so $-C_i^4=C_{i+2}^4$. Observe that $\mathrm{Int}(\mathcal{D})$ is 
	\[
	\cup_{i \in C_0^4} \Delta(D_i)
	=
	\cup_{i \in C_0^4} \{(1-\gamma)i,-(1-\gamma)i\}
	=
	(1-\gamma)C_0^4 \cup -(1-\gamma)C_0^4.
	\]
		If $1-\gamma \in C_0^2$, then
	$
	(1-\gamma)C_0^4 \cup -(1-\gamma)C_0^4
	=
	(1-\gamma)C_0^4 \cup (1-\gamma)C_2^4
	=
	(1-\gamma)(C_0^4 \cup C_2^4)
	=
	(1-\gamma)C_0^2
	=
	C_0^2,
	$ 
	so $\mathrm{Int}(\mathcal{D})$ is $C_0^2$ in this case.
		Similarly, if $1-\gamma \notin C_0^2$, then $1-\gamma \in C_1^2$, so
	$
	(1-\gamma)C_0^4 \cup -(1-\gamma)C_0^4
=	C_1^2,
	$ 
	and $\mathrm{Int}(\mathcal{D})$ is $C_1^2$ . 
	
	Now
	$
	\cup_{ i \in C_0^4} D_i= C_0^4 \cup \gamma C_0^4.
	$ 
	Since $\gamma \in C_2^4$, we have $\gamma C_0^4=C_2^4$, so
	$
	\cup_{i \in C_0^4 } D_i=C_0^4 \cup C_2^4=C_0^2.
	$ 
	
	Since $q \equiv 1 \mod 4$, $\Delta(C_0^2)$ comprises $(q-5)/4$ copies of $C_0^2$ and $(q-1)/4$ copies of $C_1^2$.  If $1-\gamma \in C_0^2$, then $\mathrm{Int}(\mathcal{D})$ is $C_0^2$, so $\mathrm{Ext}(\mathcal{D})$ comprises $(q-9)/4$ copies of $C_0^2$ and $(q-1)/4$ copies of $C_1^2$.  If $1-\gamma \notin C_0^2$, then $\mathrm{Int}(\mathcal{D})$ is $C_1^2$, so $\mathrm{Ext}(\mathcal{D})$ comprises $(q-5)/4$ copies of $C_0^2$ and $(q-5)/4$ copies of $C_1^2$. This completes the proof.
\end{proof}
\begin{example}
For $q=13$, taking primitive element $2$, we have $C_0^4=\{1,3,9\}$.  Take $\gamma=4 \in C_2^4$; then $1-\gamma=-3=10 \in C_0^2$ and so $\mathcal{D}=\{ \{1,4\}, \{3,12\}, \{9,10\}\}$ is a $(13,3,2;1,3)$-EPDF and a $(13,3,2;1,0)$-DPDF.
\end{example}

\begin{remark}
	Taking $\gamma=-1$ in~\ref{thm:q5mod8epdf}, we obtain the family
	$
	D=\{D_i : i \in C_0^4\},
	$ 
	where $D_i=\{i,-i\}$. In this case, the sets $D_i$ are precisely the even cyclotomic classes of order $(q-1)/2$ partitioning $C_0^2$; see Theorem~5.13(ii) of~\cite{HucJoh}. For other choices of $\gamma$, the sets $D_i$ need not be cyclotomic classes.
\end{remark}

\begin{theorem}\label{thm:q1mod8epdf}
	Let $q$ be a prime power congruent to $1$ modulo $8$. Let $\alpha$ be a primitive element of $\mathbb{F}_q$. Let $\gamma \in C_2^4$ and define
	$
	D_i=\{i,-i,\gamma i,-\gamma i\}$ $  (i \in R),
	$ 
	where $R$ is a set of representatives for the quotient group $C_0^4/\{1,-1\}$. Let $\mathcal{D}=\{D_i: i \in R\}$. Then $\mathcal{D}$ is a
	$
	(q,(q-1)/8,4;(q-17)/4,(q-1)/4)\text{-EPDF}
	$ 
	and a
	$ 
	(q,(q-1)/8,4;3,0)\text{-DPDF},
	$ 
	if one of $1-\gamma$ and $1+\gamma$ lies in $C_0^4$ while the other lies in $C_2^4$.
\end{theorem}

\begin{proof}
	Since $q \equiv 1 \mod 8$, we have $-1=\alpha^{(q-1)/2}\in C_0^4$, so $-C_i^4=C_i^4$ for each $i$. Also, since $q \equiv 1 \mod 8$, by Lemma \ref{lem:2is} we have $2 \in C_0^2$. Hence $2 \in C_r^4$ for some $r \in \{0,2\}$.
	
	Observe that 
	$
	\cup_{i\in R}D_i=C_0^4\cup \gamma C_0^4.
	$ 
	Since $\gamma \in C_2^4$, we have $\gamma C_0^4=C_2^4$, so
	$
	\cup_{i\in R}D_i=C_0^4\cup C_2^4=C_0^2.
	$ 
	
	We now consider the multiset of internal differences of $\mathcal{D}$. We get
	\[
	\bigcup_{i\in R}\Delta(D_i)
	=
	\bigcup_{i\in R}\{2i,-2i,2\gamma i,-2\gamma i\}
	\cup
	2\bigcup_{i\in R}\{(1-\gamma)i,-(1-\gamma)i\}
	\cup
	2\bigcup_{i\in R}\{(1+\gamma)i,-(1+\gamma)i\}.
	\]
	Since $R$ is a set of representatives for $C_0^4/\{1,-1\}$ and $-1\in C_0^4$, it follows that
	$
	\cup_{i\in R}\{2i,-2i\}=2C_0^4=C_r^4
	$ 
	and
	$
	\cup_{i\in R}\{2\gamma i,-2\gamma i\}=2\gamma C_0^4=C_{r+2}^4.
	$ 
	As $\{r,r+2\}=\{0,2\}$, we have
	\[
	\bigcup_{i\in R}\{2i,-2i,2\gamma i,-2\gamma i\}=C_0^4 \cup C_2^4=C_0^2.
	\]
	
	If
	$
	1-\gamma \in C_0^4$ and $  1+\gamma \in C_2^4,
	$ 
	then
	\begin{align*}
		2\bigcup_{i\in R}\{(1-\gamma)i,-(1-\gamma)i\} = 2C_0^4,  \quad \text{ and }
		\quad 2\bigcup_{i\in R}\{(1+\gamma)i,-(1+\gamma)i\} = 2C_2^4.
	\end{align*}
	Similarly, if
	$
	1-\gamma \in C_2^4,$ and $ 1+\gamma \in C_0^4,
	$ 
	\begin{align*}
		2\bigcup_{i\in R}\{(1-\gamma)i,-(1-\gamma)i\} = 2C_2^4,  \quad \text{ and }
		\quad 2\bigcup_{i\in R}\{(1+\gamma)i,-(1+\gamma)i\} = 2C_0^4.
	\end{align*}
	In either case,
	$
	\bigcup_{i \in R} \Delta(D_i)=3C_0^2.
	$ 	Thus $\mathcal{D}$ is a $(q,(q-1)/8,4;3,0)$-DPDF.

	Since $q \equiv 1 \mod 4$, the difference multiset of $C_0^2$ comprises $(q-5)/4$ copies of $C_0^2$ and $(q-1)/4$ copies of $C_1^2$. Subtracting the internal difference multiset, the external difference multiset of $\mathcal{D}$ comprises $(q-17)/4$ copies of $C_0^2$ and $(q-1)/4$ copies of $C_1^2$. Hence $\mathcal{D}$ is a $(q,(q-1)/8,4;(q-17)/4,(q-1)/4)$-EPDF. This completes the proof.
\end{proof}

\begin{remark}
In \cite{ChaDin}, it is shown that when $q \equiv 1 \mod 8$, the family $\mathcal{D}= \{D_i : i \in R\}$, where $D_i=\{i,-i,\gamma i,-\gamma i\}$, $\gamma \in C_2^4$ and $R$ is a set of representatives for $C_0^4/\{1,-1\}$, is an $(q,(q-1)/8,4,(q-9)/4)$ -EDF whenever $1-\gamma \in C_1^4$  and $  1+\gamma \in C_3^4$.  By the argument in the proof of Theorem \ref{thm:q1mod8epdf}, the same conclusion holds when
$1-\gamma \in C_3^4$  and $ 1+\gamma \in C_1^4$. 
\end{remark}

Since for any odd $q$, $C_0^{(q-1)/2}=\{1,-1\}$, each $D_i$ of Theorem \ref{thm:q1mod8epdf} is the union $D_i=C_a^{(q-1)/2} \cup C_{a+b}^{(q-1)/2}$ where $i=\alpha^a$ and $\gamma=\alpha^b$.  This is a single class precisely if $b=(q-1)/4$.
We give examples of Theorem \ref{thm:q1mod8epdf} below.
\begin{example}
\begin{itemize}
\item[(i)] In $\mathbb{F}_{25}$, let primitive element $\alpha$ be a root of $x^2+2x+3 \in \mathbb{F}_5[x]$.  Then $C_0^4=\{1,4\alpha+2,4\alpha+1,4,\alpha+3,\alpha+4\}$ and take $R=\{1,\alpha+3,\alpha+4\}$.  Let $\gamma=\alpha^6=3$; then $1-\gamma=3 \in C_2^4$ and $1+\gamma=4 \in C_0^4$. Here $D_1=\{1,2,3,4\}, D_{\alpha+3}=\{\alpha+3,2\alpha+1,3\alpha+4,4\alpha+2\}$ and $D_{\alpha+4}=\{\alpha+4,2\alpha+3,3\alpha+2,4\alpha+1\}$.  Note that $(q-1)/4=6$, and $\mathcal{D}$ comprises the even cyclotomic classes of order $6$.
\item[(ii)] In $\mathbb{F}_{89}$, take primitive element $\alpha=3$.
	Let $\gamma=\alpha^2=9$. Then
	$
	1-\gamma=81=3^4 \in C_0^4
	$ 
	and
	$
	1+\gamma=10=3^{86} \in C_2^4
	$.
	Take
	$
	R=\{3^{4j}:0\leq j\leq 10\}
	$. For $i=3^{4j}\in R$, consider
	$
	D_i=\{i,-i,9i,-9i\}.
	$ 
	Since $-1=3^{44}$, we have
	$D_i=\{3^{4j},3^{4j+2},3^{4j+44},3^{4j+46}\}$.  Here $D_i=C_{4j}^{44} \cup C_{4j+2}^{44}$, so each $D_i$ is a union of two cyclotomic classes of order $(q-1)/2$ (but not a single cyclotomic class). 
	\end{itemize}
\end{example}

\section{Open Problems}
In this paper, we have shown that the skew PDSs introduced in \cite{AnbKalMei} are not an isolated phenomenon, but that they can be constructed as unions of cyclotomic classes in finite fields.  Moreover, we have demonstrated that there exist skew PDSs not only of Latin square type, but also of Paley type.  

This leads to the following natural open problems.
\begin{description}
\item[Open Problem 1]  Are there other families of skew PDSs which can be constructed using unions of cyclotomic classes in $\mathbb{F}_q$?   

\item[Open Problem 2] Do there exist skew PDSs of types other than Latin square and Paley type - for example, of negative Latin square type?

\item[Open Problem 3] We have seen constructions of skew PDSs using bent partitions \cite{AnbKalMei} and cyclotomy.  Can skew PDSs be constructed via other techniques?
\end{description}

Another contribution of this paper is to present cyclotomic constructions for families of DPDFs and EPDFs relative to a set $T$, where in our setting $T$ is a Paley PDS.   
\begin{description}
\item[Open Problem 4] Are there other families of DPDFs and EPDFs, relative to PDSs, comprising cyclotomic classes or their unions, or as partitions of cyclotomic classes?  
\item[Open Problem 5] Find families of DPDFs and EPDFs, relative to PDSs which are not of Paley type.
\end{description}

\section{Acknowledgement}
This work was carried out in part  during a research visit of Tekgül Kalayc\i \ to the University of St Andrews. Tekgül Kalayc\i \ would like to thank the School of Mathematics and Statistics at the University of St Andrews for the hospitality and support.

\section*{Appendix}
In this Appendix, we present for the reader's convenience the necessary information on cyclotomic numbers of order $4$ and $8$.

The following result due to Katre and Rajwade \cite{KatRaj} resolves the sign ambiguity for cyclotomic numbers of order $4$ for $\mathbb{F}_q$ ($q$ prime power).

\begin{theorem}\label{thm:KaRaThm}
Let $p$ be an odd prime, let $q=p^m \equiv 1 \mod 4$, and write $q=4f+1$.  Let $\alpha$ be a generator of $\mathbb{F}_q^*$. 
\begin{itemize}
    \item If $p \equiv 3 \mod 4$, let $s=(-p)^{m/2}$ and $t=0$.
    \item If $p \equiv 1 \mod 4$, define $s$ uniquely by $q=s^2+t^2$, $p \nmid s$, $s \equiv 1 \mod 4$, then $t$ uniquely by $\alpha^{(q-1)/4} \equiv s/t \mod p$.
\end{itemize}
Then the cyclotomic numbers of order $4$ for $\mathbb{F}_q$, corresponding to $v$, are determined unambiguously by the following formulae:
\begin{itemize}
    \item For $f$ even: 
    \begin{itemize}
        \item $A=(0,0)_4=\frac{1}{16}(q-11-6s)$;
        \item $B=(1,0)_4=(0,1)_4=(3,3)_4=\frac{1}{16}(q-3+2s+4t)$;
        \item $C=(2,0)_4=(0,2)_4=(2,2)_4=\frac{1}{16}(q-3+2s)$; 
        \item $D=(3,0)_4=(0,3)_4=(1,1)_4=\frac{1}{16}(q-3+2s-4t)$; 
        \item $E=(1,2)_4=(1,3)_4=(2,1)_4=(2,3)_4=(3,1)_4=(3,2)_4=\frac{1}{16}(q+1-2s)$.
    \end{itemize}
    \item For $f$ odd: 
    \begin{itemize}
        \item $A=(0,0)_4=(2,0)_4=(2,2)_4=\frac{1}{16}(q-7+2s)$; 
        \item $B=(0,1)_4=(1,3)_4=(3,2)_4=\frac{1}{16}(q+1+2s-4t)$;
        \item $C=(0,2)_4=\frac{1}{16}(q+1-6s)$;
        \item $D=(0,3)_4=(1,2)_4=(3,1)_4=\frac{1}{16}(q+1+2s+4t)$;
        \item $E=(1,0)_4=(1,1)_4=(2,1)_4=(2,3)_4=(3,0)_4=(3,3)_4=\frac{1}{16}(q-3-2s)$.
    \end{itemize}
\end{itemize}
\end{theorem}

For the cyclotomic numbers of order $8$, we have the following, from the Appendix of \cite{BraMei} and Lemma 30 of \cite{Sto}.

\begin{theorem}\label{thm:e=8cyclo}
Let $q=8f+1$ be a prime power.  The cyclotomic numbers are uniquely determined by $x,y,a$ and $b$, defined as follows.
\begin{itemize}
    \item $q=x^2+4y^2$, $x \equiv 1 \mod 4$ is the unique proper representation of $q=p^m$ if $p \equiv 1 \mod 4$; otherwise $x=\pm p^{m/2}$ and $y=0$.
    \item $q=a^2+2b^2$, $a \equiv 1 \mod 4$, is the unique proper representation of $q=p^m$ if $p \equiv 1$ or $3 \mod 8$; otherwise $a= \pm p^{m/2}$ and $b=0$. 
\end{itemize}
The signs of $y$ and $b$ are ambiguously determined.

\begin{itemize}
\item[]{\bf Case when $q=8f+1$, $f$ odd (i.e $q \equiv 9 \mod 16$)}\\
The relations between the cyclotomic numbers are given by:
\begin{itemize}
\item[]$A=(0,0)_8 =(4,0)_8 =(4,4)_8$; $B=(0,1)_8 =(3,7)_8 =(5,4)_8$;
\item[]$C=(0,2)_8 =(2,6)_8 =(6,4)_8$; $D=(0,3)_8 =(1,5)_8 =(7,4)_8$; 
\item[]$E=(0,4)_8$; $F=(0,5)_8 =(1,4)_8 =(7,3)_8$; 
\item[]$G=(0,6)_8 =(2,4)_8 =(6,2)_8$; $H=(0,7)_8 =(3,4)_8 =(5,1)_8$;
\item[]$I=(1,0)_8 =(3,3)_8 =(4,1)_8 =(4,5)_8 =(5,0)_8 =(7,7)_8$;
\item[]$J=(1,1)_8 =(3,0)_8 =(4,3)_8 =(4,7)_8 =(5,5)_8 =(7,0)_8$;
\item[]$K=(1,2)_8 =(2,7)_8 =(3,6)_8 =(5,3)_8 =(6,5)_8 =(7,1)_8$;
\item[]$L=(1,3)_8 =(1,6)_8 =(2,5)_8 =(6,3)_8 =(7,2)_8 =(7,5)_8$;
\item[]$M=(1,7)_8 =(2,3)_8 =(3,5)_8 =(5,2)_8 =(6,1)_8 =(7,6)_8$;
\item[]$N=(2,0)_8 =(2,2)_8 =(4,2)_8 =(4,6)_8 =(6,0)_8 =(6,6)_8$;
\item[]$O=(2,1)_8 =(3,1)_8 =(3,2)_8 =(5,6)_8 =(5,7)_8 =(6,7)_8$.
\end{itemize}

$
\begin{array}{|c|c|c|}
\hline
q \equiv 9 \mod 16 &\text{$2$ a quartic residue} &\text{$2$ not a quartic residue}\\
\hline
64A & q-15-2x & q-15-10x-8a \\
64B & q+1+2x-4a+16y & q+1+2x-4a-16b\\
64C & q+1+6x+8a-16y& q+1-2x+16y\\
64D & q+1+2x-4a-16y & q+1+2x-4a-16b\\
64E & q+1-18x & q+1+6x+24a\\
64F & q+1+2x-4a+16y & q+1+2x-4a+16b\\
64G & q+1+6x+8a+16y & q+1-2x-16y\\
64H & q+1+2x-4a-16y & q+1+2x-4a+16b\\
64I & q-7+2x+4a & q-7+2x+4a+16y\\
64J & q-7+2x+4a & q-7+2x+4a-16y\\
64K & q+1-6x+4a+16b & q+1+2x-4a\\
64L & q+1+2x-4a & q+1-6x+4a\\
64M & q+1-6x+4a-16b & q+1+2x-4a\\
64N & q-7-2x-8a & q-7+6x\\
64O & q+1+2x-4a & q+1-6x+4a\\
\hline
\end{array}
$

\item[]\noindent{\bf Case when $q=8f+1$, $f$ even (i.e $q \equiv 1 \mod 16$)}\\
The relations between the cyclotomic numbers are given by:
\begin{itemize}
\item[]$A=(0,0)_8$; $B=(0,1)_8 =(1,0)_8 =(7,7)_8$;
\item[]$C=(0,2)_8 =(2,0)_8 =(6,6)_8$; $D=(0,3)_8 =(3,0)_8 =(5,5)_8$; 
\item[]$E=(0,4)_8=(4,0)_8=(4,4)_8$; $F=(0,5)_8 =(5,0)_8 =(3,3)_8$; 
\item[]$G=(0,6)_8 =(6,0)_8 =(2,2)_8$; $H=(0,7)_8 =(7,0)_8 =(1,1)_8$;
\item[]$I=(1,2)_8 =(2,1)_8 =(1,7)_8 =(7,1)_8 =(6,7)_8 =(7,6)_8$;
\item[]$J=(1,3)_8 =(3,1)_8 =(2,7)_8 =(7,2)_8 =(5,6)_8 =(6,5)_8$;
\item[]$K=(1,4)_8 =(4,1)_8 =(3,7)_8 =(7,3)_8 =(4,5)_8 =(5,4)_8$;
\item[]$L=(1,5)_8 =(5,1)_8 =(3,4)_8 =(4,3)_8 =(4,7)_8 =(7,4)_8$;
\item[]$M=(1,6)_8 =(6,1)_8 =(2,3)_8 =(3,2)_8 =(5,7)_8 =(7,5)_8$;
\item[]$N=(2,4)_8 =(4,2)_8 =(2,6)_8 =(6,4)_8 =(4,6)_8 =(6,2)_8$;
\item[]$O=(2,5)_8 =(5,2)_8 =(3,5)_8 =(5,3)_8 =(3,6)_8 =(6,3)_8$.
\end{itemize}

$
\begin{array}{|c|c|c|}
\hline
q \equiv 1 \mod 16 &\text{$2$ a quartic residue} &\text{$2$ not a quartic residue}\\
\hline
64A & q-23-18x-24a & q-23+6x \\
64B & q-7+2x+4a+16y+16b & q-7+2x+4a\\
64C & q-7+6x+16y& q-7-2x-8a-16y\\
64D & q-7+2x+4a-16y+16b & q-7+2x+4a\\
64E & q-7-2x+8a & q-7-10x\\
64F & q-7+2x+4a+16y-16b & q-7+2x+4a\\
64G & q-7+6x-16y & q-7-2x-8a+16y\\
64H & q-7+2x+4a-16y-16b & q-7+2x+4a\\
64I & q+1+2x-4a & q+1-6x+4a\\
64J & q+1-6x+4a & q+1+2x-4a-16b\\
64K & q+1+2x-4a  & q+1+2x-4a+16y\\
64L & q+1+2x-4a & q+1+2x-4a-16y\\
64M & q+1-6x+4a & q+1+2x-4a+16b\\
64N & q+1-2x & q+1+6x+8a\\
64O & q+1+2x-4a & q+1-6x+4a\\
\hline
\end{array}
$

\end{itemize}

\end{theorem}

\end{document}